\documentclass[11pt]{amsart} 
\usepackage{amsmath, amssymb, latexsym, tikz}
\usepackage{caption,ifthen,url,fancyhdr,cite} 
\usepackage[foot]{amsaddr}     
\usepackage{libertine}      
\usepackage{longtable}   
\usepackage{upquote} 
\usepackage{listings} 
\usepackage{textcomp}
\usepackage{adjustbox}
\usepackage{soul}
 
\newcommand\textcode[1]{{\footnotesize \rm\texttt{#1}}}   

\makeatletter 
\renewcommand\subsubsection{\@startsection{subsubsection}{3}%
  \z@{.5\linespacing\@plus.7\linespacing}{-.5em}%
  {\normalfont\bfseries}} 
\makeatother
 
\usepackage[T1]{fontenc} 

\newcommand\makequotestraight{%
\begingroup\lccode`~=`' 
\lowercase{\endgroup\let~}\textquotesingle
\catcode`'=\active
}

\usepackage[a4paper,right=2.5cm, left=2.5cm, top=2.5cm, bottom=2.5cm, marginpar=1.6cm]{geometry}

\newtheoremstyle{mythm}                   
{6pt}
{6pt}
{\it}
{}
{\bf}
{.}
{.5em}
{}

\newtheoremstyle{mydef}                   
{6pt}
{6pt}
{}
{}
{\bf}
{.}
{.5em}
{}

\newtheoremstyle{myrem}                   
{6pt}
{6pt}
{}
{}
{\bf}
{.}
{.5em}
{}

\theoremstyle{mythm}      
\newtheorem{theorem}{Theorem}[section]
\newtheorem{proposition}[theorem]{Proposition}
\newtheorem{lemma}[theorem]{Lemma}

\theoremstyle{mydef}

\theoremstyle{myrem}
\newtheorem{remark}[theorem]{Remark}
\numberwithin{equation}{section}

\newcounter{ithmcount}

\newenvironment{iprf}{\begin{list}{{\rm
	\alph{ithmcount})}}{\usecounter{ithmcount}\labelwidth-5pt
      \leftmargin0pt \topsep3pt \itemsep1pt \parsep2pt}}{\end{list}}

\renewcommand{\leq}{\leqslant} 

\newcommand{\PSL}{{\rm PSL}}
\newcommand{\PGL}{{\rm PGL}}
\newcommand{\PSU}{{\rm PSU}}

\newcommand{\GL}{{\rm GL}} 
\newcommand{\MM}{\mathbf{M}}
\newcommand{\BB}{\mathbf{B}}
\newcommand{\GG}{\mathbf{G}}
\newcommand{\QQ}{\mathbf{Q}}
\newcommand{\Gx}{2^{1+24}\udot{\rm Co}_1}

\newcommand{\mt}[1]{{\rm #1}}

\makeatletter
\newcommand{\udot}{\mathpalette\udot@\relax}
\newcommand{\udot@}[2]{%
  \begingroup
  \sbox\z@{$#1{:}$}%
  \sbox\tw@{$#1{.}$}%
  \raisebox{\dimexpr\ht\z@-\ht\tw@}{$\m@th#1.$}%
  \endgroup
}
\makeatother

\usepackage{fancyvrb}
\fvset{commandchars=\\\{\}}

\makeatletter
\@namedef{subjclassname@1991}{2020 MSC}
\makeatother

\begin{document}

\vspace*{-0.5cm}

\title{The maximal subgroups of the Monster}
\subjclass[2000]{} 
\author[H. Dietrich]{Heiko Dietrich}
\author[M. Lee]{Melissa Lee}
\author[T. Popiel]{Tomasz Popiel}
\address[Dietrich, Lee, Popiel]{School of Mathematics, Monash University, Clayton VIC 3800, Australia}
\email{\rm heiko.dietrich@monash.edu, melissa.lee@monash.edu, tomasz.popiel@monash.edu}
\date{\today}

\begin{abstract}
The classification of the maximal subgroups of the Monster $\MM$ is a long-standing problem in finite group theory. 
According to the literature, the classification is complete apart from the question of whether $\MM$ contains maximal subgroups that are almost simple with socle $\PSL_2(13)$. 
However, this conclusion relies on reported claims, with unpublished proofs, that $\MM$ has no maximal subgroups that are almost simple with socle $\PSL_2(8)$, $\PSL_2(16)$, or $\PSU_3(4)$. 
The aim of this paper is to settle all of these questions, and thereby complete the {solution to the maximal subgroup problem for $\MM$, and  for the sporadic simple groups as a whole.}
{Specifically, we prove the existence of} two new maximal subgroups of $\MM$, isomorphic to the automorphism groups of $\PSL_2(13)$ and $\PSU_3(4)$, and we {establish} that $\MM$ has no almost simple maximal subgroup with socle $\PSL_2(8)$ or $\PSL_2(16)$. {We also} correct the claim that $\MM$ has no almost simple maximal subgroup with socle $\PSU_3(4)$, {and provide evidence that the maximal subgroup $\PSL_2(59)$ (constructed in 2004) does not exist}. Our proofs are supported by reproducible computations carried out using the publicly available Python package \textcode{mmgroup} for computing with $\MM$ recently developed by M.~Seysen. {We provide} explicit generators for our newly discovered maximal subgroups of $\MM$ {in \textcode{mmgroup} format.}
\end{abstract}

\maketitle


{\centering \it Dedicated to Rob Wilson for his many contributions to the study of the sporadic finite simple groups. \par}

\thispagestyle{empty}

\section{Introduction}\label{sec_intro} 

\noindent The Monster $\MM$ is the largest of the 26 sporadic finite simple groups, and the only sporadic group for which the classification of the maximal subgroups remains open. 
The classification of the maximal subgroups of $\MM$ is mostly complete, due chiefly to significant efforts of S.\ P.\ Norton, R.\ A.\ Wilson, and Wilson's former student P.\ E.\ Holmes. 
As Wilson explains in his contribution to the proceedings of the ``Thirty years of the Atlas'' conference \cite[pp.~64--65]{W17}, at least 15 papers had been written on this topic as of 2017, but it remained to classify the (possible) maximal subgroups of $\MM$ that are almost simple with socle isomorphic to one of $\PSL_2(8)$, $\PSL_2(13)$, $\PSL_2(16)$, $\PSU_3(4)$, or $\PSU_3(8)$.  
These cases seemed largely resistant to theoretical arguments, while also posing significant computational difficulties. 
Nevertheless, Wilson \cite{W17b} then ruled out the case $\PSU_3(8)$, and reported that $\PSL_2(8)$, $\PSL_2(16)$, and $\PSU_3(4)$ had been eliminated in unpublished work; see \cite[p.~65]{W17} and \cite[p.~877]{W17b}. 
Based on this, it would seem that the only remaining task is to determine, up to conjugacy, all maximal subgroups of $\MM$ that are almost simple with socle $\PSL_2(13)$. 
However, it also seems prudent to check the unpublished cases $\PSL_2(8)$, $\PSL_2(16)$, and $\PSU_3(4)$. 
The aim of this paper is to settle all four cases, and thereby complete the classification of the maximal subgroups of $\MM$. 
Our results are as follows. 

\begin{theorem} \label{thmL213}
The Monster has a unique conjugacy class of maximal subgroups with socle isomorphic to $\PSL_2(13)$. 
These subgroups are isomorphic to $\PSL_2(13){:}2$ and have non-trivial intersection with precisely the conjugacy classes $2\mt{B}$, $3\mt{B}$, $4\mt{C}$, $6\mt{E}$, $7\mt{B}$, $12\mt{H}$, $13\mt{A}$, and $14\mt{C}$ of $\MM$.
\end{theorem}

\begin{theorem} \label{thmU34}
The Monster has a unique conjugacy class of maximal subgroups with socle isomorphic to $\PSU_3(4)$. 
These subgroups are isomorphic to $\PSU_3(4){:}4$ and have non-trivial intersection with precisely the conjugacy classes $2\mt{B}$, $3\mt{C}$, $4\mt{C}$, $4\mt{D}$, $5\mt{B}$, $6\mt{F}$, $8\mt{E}$, $10\mt{E}$, $12\mt{J}$, $13\mt{B}$, $15\mt{D}$, $16{B}$, and $16\mt{C}$ of $\MM$. 
\end{theorem}

\begin{theorem} \label{thmL216}
The Monster has no maximal subgroup that is almost simple with socle $\PSL_2(16)$. 
\end{theorem}

\begin{theorem} \label{thmL28}
The Monster has no maximal subgroup that is almost simple with socle $\PSL_2(8)$. 
\end{theorem}

{We comment on two corrections to the existing literature in Remark \ref{remPSL259}.}

Historically, various results about the Monster have appeared without detailed proofs or citations, especially in the Atlas of Finite Groups \cite{atlas}. 
This has been explicitly pointed out by Meierfrankenfeld and Shpectorov, who classified the maximal $2$-local subgroups of $\MM$; see \cite[Section~1]{MS02}. 
By now, much of the information in the Atlas has been independently checked, and lists of corrections are available; see e.g. \cite{Bnew,BMO,modularAtlas,NortonAtlasWebsite}.
However, in light of these observations, we feel that it is important to be transparent with regards to what we have taken from the literature vs. what we have independently verified. 
The following remark summarises the key existing results that we employ in our proofs.

\begin{remark} \label{rem:assumptions}
We use the result that the classification of the maximal subgroups of $\MM$ is complete apart from the possibility that there are as-yet-unclassified maximal subgroups that are almost simple with socle $S$ isomorphic to $\PSU_3(4)$ or $\PSL_2(q)$ for some $q \in \{8,13,16\}$; see \cite{W17,W17b}. 
We also apply results of Norton and/or Wilson~\cite{N98,NW02,W15} which show that, in every such maximal subgroup, the socle $S$ satisfies certain conjugacy class fusion restrictions, outlined in Section~\ref{sec_known_subgroups}. 
The proofs of Theorems~\ref{thmU34} and~\ref{thmL216} refer to Norton's classification \cite[Section~4]{N98} of the subgroups of $\MM$ isomorphic to $\mt{A}_5$. 
\end{remark}

Our proof of each of Theorems~\ref{thmL213}--\ref{thmL28} is supported by a computational construction of all conjugacy classes of subgroups of $\MM$ isomorphic to $S = \PSL_2(13)$, $\PSU_3(4)$, $\PSL_2(16)$, or $\PSL_2(8)$, subject to the conjugacy class fusion restrictions outlined in Section~\ref{sec_known_subgroups}. 
We show that every such subgroup either has non-trivial centraliser, and is therefore properly contained in some $p$-local maximal subgroup of $\MM$, or has trivial centraliser and extends to an almost simple overgroup of $S$ that is maximal in $\MM$. 
Most of our computations are carried out using Seysen's Python package \textcode{mmgroup} \cite{sey_python,sey20,sey22} for computing with $\MM$. 
In particular, we exhibit explicit generators for our new maximal subgroups $\PSL_2(13){:}2$ and $\PSU_3(4){:}4$ in {\textcode{mmgroup}} format, and make our code publicly available \cite{ourfile}. 
Based on Theorems~\ref{thmL213}--\ref{thmL28} and the existing literature, the complete classification of the maximal subgroups of $\MM$ is as follows. 

\begin{theorem}
For every group $G$ in Table~\ref{tab:allmax}, $\MM$ has a unique conjugacy class of maximal subgroups  isomorphic to $G$. 
Conversely, every maximal subgroup of $\MM$ is isomorphic to some group in Table~\ref{tab:allmax}. \end{theorem}

\renewcommand{\arraystretch}{1.12}
\newcommand{\newt}[6]{
 #1 & #2 &  & #3 & #4 &  & #5 & #6 \\
 }
\begin{table}[t]
\begin{adjustbox}{width=15.5cm,center}
\begin{tabular}{llcllcll}
\hline
\newt{$2\udot\BB$}{  }{$(7{:}3\times\mt{He}){:}2$}{  }{$(\mt{PSL}_2(11)\times \mt{PSL}_2(11)){:}4$}{  }
\newt{$2^{1+24}\udot\mt{Co}_1$}{  }{$(\mt{A}_5\times\mt{A}_{12}){:}2$}{  }{$13^2{:}2\mt{PSL}_2(13).4$}{  }
\newt{$3\udot\mt{Fi}_{24}$}{  }{$5^{3+3}\udot(2\times\mt{PSL}_3(5))$}{  }{$(7^2{:}(3\times 2\mt{A}_4)\times \mt{PSL}_2(7)).2$}{  }
\newt{$2^{2}\udot {}^2\mt{E}_6(2){:}\mt{S}_3$}{  }{$(\mt{A}_6\times\mt{A}_6\times\mt{A}_6).(2\times\mt{S}_4)$}{  }{$(13{:}6\times \mt{PSL}_3(3)). 2$}{  }
\newt{$2^{10+16}\udot\mt{P}\Omega_{10}^+(2)$}{  }{$(\mt{A}_5\times \mt{PSU}_3(8){:}3){:}2$}{  }{$13^{1+2}{:}(3\times 4\mt{S}_4)$}{  }
\newt{$2^{2+11+22}\udot(\mt{M}_{24}\times\mt{S}_3)$}{  }{$5^{2+2+4}{:}(\mt{S}_3\times\mt{GL}_2(5))$}{  }{$\mt{PSU}_3(4){:}4$}{New}
\newt{$3^{1+12}\udot 2\udot \mt{Suz}{:}2$}{  }{$(\mt{PSL}_3(2)\times\mt{PSp}_4(4){:}2)\udot 2$}{  }{$\mt{PSL}_2(71)$}{\cite{HW08}}
\newt{$2^{5+10+20}\udot(\mt{S}_3\times \mt{PSL}_5(2))$}{  }{$7^{1+4}{:}(3\times2\mt{S}_7)$}{  }{{$59{:}29^\dagger$}}{\cite{HW04}}
\newt{$\mt{S}_3\times\mt{Th}$}{  }{$(5^2{:}[2^4]\times \mt{PSU}_3(5)).\mt{S}_3$}{  }{$11^2{:}(5\times 2\mt{A}_5)$}{}
\newt{$2^{3+6+12+18}\udot (\mt{PSL}_3(2)\times 3\mt{S}_6)$}{  }{$(\mt{PSL}_2(11)\times\mt{M}_{12}){:}2$}{  }{$\mt{PSL}_2(41)$}{\cite{NW13}} 
\newt{$3^8\udot\mt{P\Omega}_8^-(3).2$}{  }{$(\mt{A}_7\times(\mt{A}_5\times \mt{A}_5){:}2^2){:}2$}{  }{$\mt{PSL}_2(29){:}2$}{\cite{HW02}}
\newt{$(\mt{D}_{10}\times\mt{HN})\udot 2$}{  }{$5^4{:}(3\times 2\udot\mt{PSL}_2(25)){:}2$}{  }{$7^2{:}\mt{SL}_2(7)$}{  }
\newt{$(3^2{:}2\times\mt{P\Omega}^+_8(3))\udot\mt{S}_4$}{  }{$7^{2+1+2}{:}\mt{GL}_2(7)$}{  }{$\mt{PSL}_2(19){:}2$}{\cite{HW08}} 
\newt{$3^{2+5+10}.(\mt{M}_{11}\times 2\mt{S}_4)$}{  }{$\mt{M}_{11}\times \mt{A}_6\udot 2^2$}{  }{$\mt{PSL}_2(13){:}2$}{New}
\newt{$3^{3+2+6+6}{:}(\mt{PSL}_3(3)\times\mt{SD}_{16})$}{  }{$(\mt{S}_5\times\mt{S}_5\times\mt{S}_5){:}\mt{S}_3$}{  }{$41{:}40$}{  }
\newt{$5^{1+6}{:}2\udot\mt{J}_2{:}4$}{  }{  }{  }{  }{  }
\hline
\end{tabular}
\end{adjustbox}
\caption{The maximal subgroups of $\MM$. The table is mostly reproduced from \cite[p.~69]{W17}. 
Where no citation is given, the group appears in the Atlas \cite{atlas} and/or \cite{MS02,MS03,W88}.\linebreak {$\dagger$ See Remark \ref{remPSL259} for an explanation of why $\mt{PSL}_2(59)$ has been replaced by $59{:}29$.}
}\label{tab:allmax}
\end{table}

Before proceeding to our proofs, we briefly discuss the role of computational group theory in the classification of the maximal subgroups of $\MM$, and the breakthrough represented by \textcode{mmgroup}. 
Initially, the classification was handled mostly theoretically, as explained in Wilson's historical account in \cite{W17}. 
Eventually, however, extensive computations were required to construct or rule out the existence of various `small' almost simple maximal subgroups; see \cite{HW02,HW04,HW08,NW13,W15}. 
Despite the rapid development of computational group theory during the same period of time (see e.g. \cite{handbook}), computing in~$\MM$ remained very difficult, because $\MM$ has no permutation or matrix representation of sufficiently small degree to facilitate efficient computations using standard algorithms. 
Many computations in $\MM$ were therefore carried out in a non-standard computational model developed by Holmes and Wilson \cite{HW03}, based on the original construction of $\MM$ due to Griess \cite{Griess}, the simplification due to Conway \cite{Conway}, and related work of Linton et~al. \cite{Linton}. 
In this model, $\MM$ is generated by \emph{standard generators} $a$ and $b$ for the centraliser $\GG\cong\Gx$ of an involution of class $2\mt{B}$, and a `triality' element $t$ that normalises the group generated by two commuting $2\mt{B}$-involutions.  
Roughly speaking, the approach is feasible because much of the computation can be done in the $24$-dimensional matrix representation of $\text{Co}_1$ over $\mathbb{F}_3$, whereas the smallest matrix representation of $\MM$ itself, over any field, is of dimension at least $196882$. 
 
Given that Holmes and Wilson (necessarily) used their own computational construction of $\MM$ as opposed to a standard computational algebra system such as {\sf GAP} \cite{gap} or {\sc Magma} \cite{magma}, it has seemed difficult for the non-expert to reproduce their computations or apply their methods to any remaining open cases of the maximal subgroup problem. 
Seysen's package \cite{sey_python,sey20,sey22} is a significant breakthrough in this regard. 
In essence, \textcode{mmgroup} is a Python implementation of Conway's construction of the Monster. 
Elements of $\MM$ are represented as words in certain natural generators, most of which are defined in terms of automorphisms of the Parker loop $\mathcal{P}$ and certain maps from $\mathcal{P}^3$ to $\mathcal{P}^3$; for details, see the section ``Representing elements of the Monster group'' in the \textcode{mmgroup} documentation \cite{sey_python}. 
Seysen's implementation is extremely efficient, due in part to a new word-shortening algorithm, described in \cite{sey22}, which allows the group operation in $\MM$ to be performed in a matter of milliseconds. 
This facilitates extensive reproducible computations that have seemingly not been possible before, with the caveat that \textcode{mmgroup} currently lacks various `basic' functionality for computing with groups that one expects from packages such as {\sf GAP} and {\sc Magma}; see Section~\ref{sec2.3}. 

\begin{remark} \label{rem:mmgroup}
As for any computer-assisted proof, our results rely on the correctness of the software that we use. 
Theoretical justification for the correctness of \textcode{mmgroup} is given in \cite{sey_python,sey20,sey22}. 
Our paper includes several code listings defining various elements of $\MM$ in \textcode{mmgroup} format, so that the reader can readily reproduce most of our proofs. 
These listings are collected in Appendix~\ref{app:listings}. 
Some of our proofs require more extensive computations, so additional supporting code (including a well-documented Jupyter Notebook file) is provided in our GitHub repository \cite{ourfile}, where we also exhibit generators for various previously known maximal subgroups of $\MM$ that we constructed during our investigations.
\end{remark}

\begin{remark}\label{remPSL259}
\begin{iprf}    
\item It was claimed (without proof) in \cite{W17,W17b}, and later cited in e.g.\ \cite{colva}, that $\MM$ has no almost simple maximal subgroup with socle $\PSU_3(4)$. Theorem~\ref{thmU34} shows that this claim is incorrect.
\item We also correct another error in the existing literature. 
Our student Anthony Pisani has used \textcode{mmgroup} to construct a copy of each of the maximal subgroups of $\MM$ in appearing in \cite[p.~69]{W17} with the exception of $\mt{PSL}_2(59)$; see \cite{maxcode}. 
According to Holmes and Wilson \cite{HW04}, every $\mt{PSL}_2(59) < \MM$ can be generated some $A<\MM$ isomorphic to $\mt{A}_5$ together with a $2\mt{B}$-involution $i\in\MM$ centralising a dihedral subgroup $\mt{D}_{10}$ of $A$. 
More precisely, $A$ must belong to one of three $\MM$-classes of subgroups isomorphic to $\mt{A}_5$ (of which there are eight in total), and for each given $A$ there are precisely $500$ candidates for $i$.
We were able to construct (copies of) each of the candidates for $A$, and, as in \cite[p.~143]{HW04}, found all candidates for $i$ in each case. 
However, none of the involutions extended any $A$ to a group isomorphic to $\mt{PSL}_2(59)$: we verified this explicitly in \textcode{mmgroup} by considering element orders. 
Our computations contradict the final conclusion of \cite{HW04} that some of the involutions should produce a subgroup $\mt{PSL}_2(59)$. 
We discussed this discrepancy with the second author of \cite{HW04}, but it seems that the code used in \cite{HW04} is no longer available. 
Based on the theory in \cite{HW04} and our reproducible computations, we conclude that the maximal subgroup $\mt{PSL}_2(59)$ of $\MM$ proposed in \cite{HW04} does not exist. 
The unique maximal subgroup of $\MM$ containing an element of order $59$ must therefore be $59{:}29$, which is what we have listed in Table~\ref{tab:allmax}.
\end{iprf}
\end{remark}\enlargethispage{0.8cm}

\noindent {\bf Acknowledgements.} 
The second author  acknowledges the support of an Australian Research Council Discovery Early Career Researcher Award (project number DE230100579). 
We thank Martin Seysen for advice regarding his Python package \textcode{mmgroup}; Rob Wilson for clarification regarding the status of the maximal subgroup problem for the Monster, and for discussions regarding the $\PSL_2(8)$ and $\PSL_2(59)$ cases; Thomas~Breuer for several comments on earlier versions of the paper; Frank L\"ubeck and Klaus Lux for helpful discussions; and Gerald H\"ohn for sharing the Python code used in the proof of Proposition~\ref{prop:U34_fusions}. 
We also thank an anonymous referee for their helpful comments.


\section{Preliminaries} \label{sec_prel} 

\noindent Most of our group-theoretic notation is standard, and usually follows the Atlas \cite{atlas}, with the notable exception that we write $\PSL_d(q)$ instead of $\mt{L}_d(q)$ and adopt similar notation for other simple classical groups. 
We denote by $\mt{D}_n$, $\mt{A}_n$, and $\mt{S}_n$ the dihedral group of order~$n$, the alternating group of degree $n$, and the symmetric group of degree $n$. 
We use $n$ to denote a cyclic group of order $n$, and $[n]$ to denote an unspecified group of order $n$. 
An extension of a group $B$ by a group $A$ is denoted by $A.B$ (or sometimes $AB$), where $A$ is the normal subgroup. 
Sometimes the notation $A{:}B$ is used to highlight that an extension splits, and $A\udot B$ denotes a non-split extension. 
The notation $p^k$ means an elementary abelian group of order $p^k$ for $p$ a prime and $k$ a positive integer, and $p^{k+\ell}$ denotes an extension $p^k.p^\ell$. 
We often use a subscript to indicate the order of a group element; for example, $g_5$ might denote an element of order $5$.

\subsection{Conjugacy classes in the Monster}\label{sec_MM}

The Monster $\MM$ has two conjugacy classes of involutions, denoted $2\mt{A}$ and $2\mt{B}$ in the Atlas. 
The respective involution centralisers are maximal subgroups of $\MM$ of the form $2 \udot \BB$, the double cover of the Baby Monster $\BB$, and $\Gx$, where $\mt{Co}_1$ is Conway's first sporadic simple group. 
We fix $z\in 2\mt{B}$ and write 
\[
\GG =C_\MM(z) \cong \Gx,
\] 
although we may sometimes abuse notation and write $\GG$ for some unspecified conjugate of $\Gx$. 
As explained in Section~\ref{sec2.3}, the \textcode{mmgroup} package includes special functionality for computing in a certain distinguished conjugate of $\GG$, namely the centraliser of the involution defined by the \textcode{mmgroup} labelling $\textcode{M<x\_1000h>}$. 
We therefore have in mind that our fixed $z$ is this specific involution. 

We also recall that the conjugacy classes of $\MM$ include exactly three classes of order-$3$ elements, $3\mt{A}$--$3\mt{C}$; two classes of order-$5$ elements, $5\mt{A}$ and $5\mt{B}$; six classes of order-$6$ elements, $6\mt{A}$--$6\mt{F}$; two classes of order-$7$ elements, $7\mt{A}$ and $7\mt{B}$; and two classes of order-$13$ elements, $13\mt{A}$ and $13\mt{B}$. 
The subgroup $\GG$ of $\MM$ has exactly seven classes of involutions, four classes of order-$3$ elements, three classes of order-$5$ elements, $22$ classes of order-$6$ elements, two classes of order-$7$ elements, and one class of order-$13$ elements. 
The group $\mt{Co}_1$ has exactly three classes of involutions, four classes of order-$3$ elements, three classes of order-$5$ elements, nine classes of order-$6$ elements, two classes of order-$7$ elements, and one class of order-$13$ elements. 
It is known how conjugacy classes are fused when comparing $\GG$ with its quotient $\mt{Co}_1$ and its overgroup $\MM$. 
This information is stored in {\sf GAP} \cite{GAPbc,gap}, and can be accessed as demonstrated in Listing~\ref{figgap1} in Appendix~\ref{app:listings}; for example, ``Output 1'' of Listing~\ref{figgap1} indicates that the (unique) class $13\mt{A}$ of order-$13$ elements in $\GG$ fuses to the class $13\mt{B}$ in $\MM$.

\subsection{Known subgroups of $\MM$ isomorphic to $\PSU_3(4)$ or $\PSL_2(q)$, $q \in \{8,13,16\}$} \label{sec_known_subgroups}

Per Remark~\ref{rem:assumptions}, the Monster contains various subgroups isomorphic to $\PSU_3(4)$ or $\PSL_2(q)$, $q \in \{8,13,16\}$, and it is known that any `new' such subgroup of $\MM$ must satisfy certain conjugacy class fusion restrictions.

Norton~\cite{N98} identified two non-conjugate subgroups of $\MM$ isomorphic to $\PSL_2(13)$.  
There is a $\PSL_2(13)$ with centraliser $3^{1+2}.2^2$ and normaliser contained in $(3^{1+2}.2^2\times {\rm G}_2(3)).2$, and a $\PSL_2(13)$ with centraliser of order $3$ and normaliser contained in $3\udot {\rm Fi}_{24}$ (the normaliser of a $3\mt{A}$-element). 
These subgroups are not maximal in $\MM$, and they do not extend to almost simple maximal subgroups. 
These results are stated more explicitly in \cite[Theorem 1]{W15}.
Norton and Wilson then showed that in every {as-yet-unclassified} $\PSL_2(13) < \MM$, elements of orders $2$, $7$, and $13$ must lie in the $\MM$-classes $2\mt{B}$, $7\mt{B}$, and $13\mt{A}$; see \cite[Table~3]{NW02} and \cite[Theorem~2]{W15}. 
Norton~\cite[Section~5]{N98} also identified three classes of $\PSL_2(8) < \MM$, all of which have non-trivial intersection with precisely the $\MM$-classes $2\mt{B}$, $3\mt{B}$, and $7\mt{A}$. 
These $\PSL_2(8)$ are not maximal in $\MM$, and they do not extend to almost simple maximal subgroups. 
By \cite[Table~3]{NW02}, every as-yet-unclassified $\PSL_2(8) < \MM$ must meet the $\MM$-classes $2\mt{B}$, $3\mt{B}$, and $7\mt{B}$.

Norton and Wilson~\cite{N98,NW02} have classified the subgroups of $\MM$ isomorphic to $\PSL_2(16)$ or $\PSU_3(4)$ whose elements of order $5$ belong to the $\MM$-class $5\mt{A}$. 
These subgroups are listed in \cite[Table~5]{N98}. 
There is one conjugacy class of each of $\PSL_2(16)$ and $\PSU_3(4)$ in $\MM$ containing $5\mt{A}$-elements. 
These groups are not maximal in $\MM$, and they do not extend to almost simple maximal subgroups. 
Every as-yet-unclassified subgroup of $\MM$ isomorphic to $\PSL_2(16)$ or $\PSU_3(4)$ must have all of its elements of order $5$ lying in the $\MM$-class $5\mt{B}$. 
By \cite[Table~3]{NW02}, such a subgroup must also satisfy certain other conjugacy class fusion restrictions. 
In particular, elements of order $2$ must belong to the $\MM$-class $2\mt{B}$, and in the case of $\PSU_3(4)$, elements of order $3$ must belong to the $\MM$-class $3\mt{C}$.

\subsection{Computing with the Monster in \textcode{mmgroup}} \label{sec2.3} 

As explained in Section~\ref{sec_intro}, \textcode{mmgroup} represents a breakthrough in computing with $\MM$, in the sense that the group operation can be performed efficiently for the first time. 
On the other hand, it currently lacks much of the functionality of computer algebra systems such as {\sf GAP} \cite{gap} or {\sc Magma} \cite{magma}, so a problem like the one that we are addressing here still poses significant technical difficulties to the user. 
We now summarise the relevant existing functionality of \textcode{mmgroup}, and explain how we overcome some of its current limitations. 

Listing~\ref{fig:mmgroup1} in Appendix~\ref{app:listings} gives some basic commands for computing in $\MM$ and its maximal subgroup $\GG \cong 2^{1+24}\udot \mathrm{Co}_1$. 
As noted in Section~\ref{sec_MM}, there is special functionality for computing in $\GG$ that does not apply to the whole of $\MM$. 
In particular, computation in $\GG$ is significantly faster than computation outside of $\GG$; see the section ``Computation in the subgroup G\_x0 of the Monster'' in \cite{sey_python}. 
Per Listing~\ref{fig:mmgroup1}, it is possible to generate random elements of $\MM$ and random elements of $\GG$, and to test for membership of an element of $\MM$ in $\GG$ and in the normal subgroup $\QQ \cong 2^{1+24}$ of $\GG$. 
Moreover, given $g \in \GG$, it is possible to compute the values of certain characters on $g$ using the method \textcode{chi\_G\_x0()}. 
This is extremely useful for distinguishing between certain conjugacy classes; see e.g. Section~\ref{sec_stdGG}. 
Given an involution $i \in \MM$, the method \textcode{conjugate\_involution()} determines which of the two $\MM$-classes of involutions $i$ belongs to, and returns an element $h$ conjugating $i$ to a `standard' element of $i^\MM$. 
If $i \in 2B$, then $i^h$ is the involution $z$ defined in Section~\ref{sec_MM}. 
This feature is of critical importance, because it allows us to compute as efficiently in the centraliser of an arbitrary $2\mt{B}$-involution as in the centraliser $\GG$ of the distinguished $2\mt{B}$-involution $z$. 
This strategy of ``changing post'' was developed and used extensively by Holmes and Wilson; see e.g. \cite[Section~1.4]{HW08} or \cite[Section~3]{black}.

\subsection{Performing certain calculations in {\sc Magma}} \label{secQ} 

There is another reason why it is preferable to carry out our calculations in the maximal subgroup $\GG$ of $\MM$ whenever possible. 

Let $\QQ$ denote the normal subgroup $2^{1+24}$ of $\GG$, and recall that $\QQ/Z(\QQ)\cong 2^{24}$ is isomorphic as a $\mt{Co}_1$-module to the Leech lattice modulo~$2$. 
In \textcode{mmgroup}, elements of $\QQ$ are labelled by the integers $0$ to $2^{25}-1$, such that elements whose labels are congruent modulo $2^{24}$ correspond to the same element in $\QQ/Z(\QQ)$. 
If \textcode{x} is such an integer, then the \textcode{mmgroup} command \textcode{MM(XLeech2(x))} returns an object representing the corresponding element in $\QQ$. 
Given $g \in \GG$, the Python function in Listing~\ref{fig:mmgroup2} in Appendix~\ref{app:listings} returns a $24\times 24$ matrix describing the action of $g$ on $2^{24}$. 
This yields a group homomorphism $\pi\colon \GG\to \GL_{24}(2)$ whose image is the $24$-dimensional Leech lattice representation of ${\rm Co}_1$. 

In Section~\ref{sec_stdGG}, we describe how to construct Holmes and Wilson's \cite{HW03} {\em standard generators} $a$~and~$b$ for $\GG$ in \textcode{mmgroup}. 
This enables us to use the matrices $A=\pi(a)$ and $B=\pi(b)$ to define $\pi(\GG) \cong \mt{Co}_1$ in {\sc Magma}. 
The {\sc Magma} algorithms resulting from the Matrix Group Recognition Project~\cite{CT} then allow us to perform many calculations in $\mt{Co}_1$. 
Given $w \in \mt{Co}_1$, the {\sc Magma} function {\sc InverseWordMap} can be used to write $w$ as a word, or, more precisely, a straight-line program (SLP), in the generators $A$ and $B$. 
(For a formal definition of SLPs, see \cite[Section~3.1.3]{handbook}.) 
Because $\pi$ maps $\{a,b\}$ to $\{A,B\}$, such SLPs allow us to compute preimages under $\pi$ efficiently: given an SLP in $\{A,B\}$ describing $w \in \mt{Co}_1$, the same SLP evaluated in $\{a,b\}$ yields $w' \in \GG$ such that $\pi(w')=w$. 
The full preimage of $w$ under $\pi$ is $w'\QQ$, and \textcode{mmgroup}'s membership test \textcode{in\_Q\_x0()} and its integer labelling of the elements in $\QQ$ allow us to modify $w'$ by elements of $\QQ$ to obtain coset representatives satisfying desired properties, if required. 
This is an efficient brute-force variant of the so-called ``applying the formula'' approach used extensively by Norton and Wilson; see e.g. \cite[Section~4.5]{NW13}.

\subsection{Finding standard generators for $\GG$}\label{sec_stdGG} 

Holmes and Wilson \cite{HW03} define standard generators for $\GG \cong \Gx$ to be elements $a$ and $b$ of order $4$ and $6$, respectively, that project to standard generators $a'$ and $b'$ of $\mt{Co}_1$ as defined in \cite{W96}. 
Specifically,  $a'$, $b'$, and $(a'b')^{-1}$ lie in the $\mt{Co}_1$-classes $2\mt{B}$, $3\mt{C}$, and $40\mt{A}$, respectively, and $a'b'a' (b')^2$ has order $6$. 
Output~2 in Listing~\ref{figgap1} shows that $6\mt{E}$ and $6\mt{F}$ are the only $\GG$-classes of order-$6$ elements that project to $3\mt{C}$ in $\mt{Co}_1$. 
This implies that $b \in 6\mt{E}\cup 6\mt{F}$. 
Similarly, $a$ and $(ab)^{-1}$ lie in $4\mt{F}\cup 4\mt{G}$ and $40\mt{F}\cup 40\mt{G}$ in $\GG$, respectively.
We constructed standard generators for $\GG$ in \textcode{mmgroup} as follows; one application of such a generating set is outlined in Remark~\ref{rem:N13}.

The aforementioned $\GG$-classes can be identified in \textcode{mmgroup} using \textcode{chi\_G\_x0()} and known class fusion information; cf. Section~\ref{sec_MM}. 
For example, Output~3 in Listing~\ref{figgap1} shows that the $\GG$-classes $4\mt{F}$ and $4\mt{G}$ are the only $\GG$-classes of order-$4$ elements on which the character $\chi_{299}$ defined in Listing~\ref{fig:mmgroup1} has value $-13$. 
The classes $6\mt{E}$, $6\mt{F}$, and $6\mt{U}$ are the only $\GG$-classes of order-$6$ elements with $\chi_{299}$-value $2$. 
They are distinguished by the character $\chi_\MM$ for $\MM$ defined in Listing~\ref{fig:mmgroup1}: $6\mt{E}$ and $6\mt{F}$ fuse to the $\MM$-classes $6\mt{E}$ and $6\mt{D}$, with $\chi_\MM$-values $5$ and $-3$, respectively, while $6\mt{U}$ fuses to the $\MM$-class $6\mt{C}$, with $\chi_\MM$-value $14$. 
We used a random search in $\GG$ to find elements that power to elements $a$ and $b$ of the form described above. 
Using the code in Listing~\ref{fig:mmgroup2} and {\sc Magma}, we confirmed that $\pi(a),\pi(b) \in \GL_{24}(2)$ generate $\mt{Co}_1$, where $\pi \colon \GG \rightarrow \text{GL}_{24}(2)$ is the homomorphism described in Section~\ref{secQ}. 
To show that $\langle a,b\rangle=\GG\cong\QQ\udot\mt{Co}_1$, we adapted the approach of \cite[Section 4.1]{NW13}. 
We verified that the elements $j_i=(a^2)^{(ab)^i}$, where $i=0,\ldots,23$, lie in $\QQ$ and form a basis of $\QQ/Z(\QQ) \cong 2^{24}$. 
We also verified that $[a^2,b^3]=z\in Z(\GG)$, so that $\GG=\langle a,b\rangle$. 
Our standard generators for $\GG$ are shown in Listing~\ref{fig:std_ab} in Appendix~\ref{app:listings}; the supporting calculations described here are documented in our GitHub repository \cite{ourfile}. 
(We have not checked whether our pair of standard generators for $\GG$ is conjugate to the pair used by Holmes and Wilson, but this property --- or lack thereof --- has no bearing on our proofs.)

\subsection{Auxiliary computational group theory algorithms} \label{ss:algo}

We have complemented \textcode{mmgroup} with our own Python implementations of standard algorithms for finding `random' elements in a subgroup $H$ of $\MM$ and for enumerating $H$, assuming that $H$ is given by a set of generators. 
We make no claims as to the efficiency of our implementations, nor to the `randomness' of our product replacement algorithm, but our code was sufficient for our purposes and  is included in \cite{ourfile}. 

The product replacement algorithm is used to obtain nearly uniformly distributed random elements in $H$; see \cite{CLMNO95} or \cite[Section 3.2.2]{handbook}. 
Given an ordered generating set $[x_1,\ldots,x_n]$ for $H$, the algorithm iteratively replaces $x_i$ by $x_ix_j^{\pm 1}$ or by $x_j^{\pm 1}x_i$ for random $i\ne j$. 
After a certain number of replacements, the new value of $x_i$ is returned. 
It is proved by Celler et al.~\cite{CLMNO95} that this process eventually produces uniformly distributed random elements of $H$.

Given a generating set $X$ for $H$, an orbit computation can be used to construct all elements of $H$, and, in particular, to determine $|H|$. 
Starting with $S=\{1\}\subseteq H$, for each $s\in S$ we run over all $x\in X$ and compute $s'=sx$. If $s'\notin S$, then $s'$ is added to $S$. 
The process terminates if, for the last $s\in S$, no new elements $sx$ have been found. 
This implies that $S$ is closed under multiplication by elements from $X$, and it follows that $S=H$; see \cite[Section 4.1]{handbook}. 
If we instead seek to construct a generating set $X$ for a subgroup of order $n$, then we can use this method to check whether $|\langle X\rangle|=n$. 
Obviously, this process can be aborted if we find that $|S|>n$. 
Given that we enumerate the group generated by $X$, this method is only feasible for reasonably small group orders.


\section{Proof of Theorem~\ref{thmL213}} \label{sec_PSLPGL}

\noindent The group $P=\PSL_2(13)$ contains (exactly) two conjugacy classes of elements of order $13$, three classes of elements of order $7$, and unique classes of elements of orders $2$, $3$, and $6$. 
The character table of $\MM$, which is available in {\sf GAP}~\cite{GAPbc,gap}, shows that for every $g \in \MM$ of order $p \in \{7,13\}$, the non-trivial powers of $g$ lie in a single conjugacy class; equivalently, the $\MM$-classes $p\mt{A}$ and $p\mt{B}$ are rational, i.e. all complex irreducible characters of $\MM$ take rational (and therefore integer) values on these classes. 
Because the Sylow $7$- and $13$-subgroups of $P$ are cyclic of prime order, it follows that in every subgroup of $\MM$ isomorphic to $P$, all elements of each order $p \in \{7,13\}$ are conjugate in $\MM$.

Up to conjugacy, the maximal subgroups of $P$ are $\mt{D}_{14}$, $\mt{D}_{12}$, $\mt{A}_4$, and $13{:}6$ where the $6$ acts faithfully. 
We seek to generate subgroups of $\MM$ isomorphic to $P$ via elements $g_{13},g_6,j_2 \in \MM$ satisfying $|g_{13}|=13$, $|g_6|=6$, $|j_2|=2$, $\langle g_{13},g_6\rangle\cong 13{:}6$, and $\langle g_6,j_2\rangle\cong \mt{D}_{12}$, the final condition being equivalent to $j_2$ inverting $g_6$ by conjugation. 
As explained in Section~\ref{sec_known_subgroups}, we can assume that $g_{13}$ is a fixed element belonging to the $\MM$-class $13\mt{A}$. 
We first exhibit such an element. 
Recall  that the normaliser of a $13\mt{A}$-element in $\MM$ is has shape $((13{:}6)\times \PSL_3(3)). 2$, while the normaliser of a $13\mt{B}$-element has shape $13^{1+2}{:}(3 \times 4\udot \mt{S}_4)$; see \cite[Section~11]{W88}. 
The respective centraliser orders are $2^4{\cdot}3^3{\cdot}13^2$ and $2^3{\cdot}3{\cdot}13^3$. 

\begin{proposition} \label{prop_13a}
The element $g_{13}$ defined in \textcode{mmgroup} format in Listing~\ref{fig:std_PGL} belongs to the $\MM$-class $13\mt{A}$. 
The elements $g_{13}$, $y_6$, $c$, and $d$ defined in Listing~\ref{fig:std_PGL} generate the index-$2$ subgroup $(13{:}6)\times \PSL_3(3)$ of $N_\MM(\langle g_{13} \rangle)$. 
Specifically, $\langle g_{13},y_6 \rangle \cong 13{:}6$, $\langle c,d \rangle \cong \PSL_3(3)$, and these two groups commute. 
\end{proposition}

\begin{proof}
A direct calculation in \textcode{mmgroup} shows that $|g_{13}|=13$ and $|y_6|=6$, and that $\langle g_{13},y_6 \rangle$ has the form $13{:}6$ and is centralised by $c$ and $d$. 
Enumerating $H=\langle c,d \rangle$ as described in Section~\ref{ss:algo} shows that it has order $5616 = 2^4{\cdot}3^3{\cdot}13 = |\PSL_3(3)|$. 
In particular, $|H|$ does not divide the order of a $13\mt{B}$-centraliser, so $g_{13} \in 13\mt{A}$. 
On the other hand, $H$ has index $13$ in $C_\MM(g_{13}) \cong 13 \times \PSL_3(3)$, and an \textcode{mmgroup} calculation confirms that $H$ does not contain $g_{13}$, so $H \cong \PSL_3(3)$.
\end{proof}

We comment in Remark~\ref{rem:N13} on how we found the element $g_{13} \in 13\mt{A}$ and the generators for the index-$2$ subgroup $N \cong (13{:}6) \times \PSL_3(3)$ of $N_\MM(\langle g_{13} \rangle)$ given in Proposition~\ref{prop_13a}.

The next step is to construct representatives of the conjugacy classes of subgroups $13\mt{A}{:}6$ of $\MM$, where the notation $13\mt{A}$ emphasises that the $13$ is of class $13\mt{A}$ (and the $6$ acts faithfully, as before). 
Up to conjugacy, every such subgroup of $\MM$ is contained in $N$. 
This can be verified in {\sf GAP} by constructing a permutation representation of $N_\MM(\langle g_{13} \rangle)$ as {\sf AtlasGroup("(13:6xL3(3)).2")}.

\begin{proposition}\label{lem_13:6}
There are exactly five conjugacy classes of subgroups of type $13\mt{A}{:}6$ in $\MM$, represented by the groups $\langle g_{13},y_6x \rangle$ where $x\in\{1,x_6,x_6^2,x_6^3,x_3\}$ with $g_{13}$, $y_6$, $x_6$, and $x_3$ defined as in Listing~\ref{fig:std_PGL}. 
The elements $g_6 = y_6x$ lie in the $\MM$-classes $6\mt{B}$, $6\mt{E}$, $6\mt{E}$, $6\mt{E}$, and $6\mt{F}$, respectively, and their cubes lie in $2\mt{B}$.
\end{proposition}

\begin{proof}
The group $H = \langle c,d \rangle \cong \PSL_3(3)$ has (exactly) five classes of elements of order dividing $6$: the identity, two classes of elements of order $3$, and unique classes of elements of orders $2$ and $6$. 
Hence, there are five $N$-classes of subgroups of type $13\mt{A}{:}6$, where $N \cong (13{:}6) \times \PSL_3(3)$ is the index-$2$ subgroup of $N_\MM(\langle g_{13} \rangle)$. 
These classes are represented by $\langle g_{13}, y_6x \rangle$, where $g_{13}$ and $y_6$ are defined as in Proposition~\ref{prop_13a}, and $x\in\{1,x_6,x_6^2,x_6^3,x_3\}$ with $x_6,x_3 \in H$ of orders $6$ and $3$, respectively, and $x_3$ not conjugate to $x_6^2$ in $H$. 
We need to show that the elements $x_6$ and $x_3$ given in Listing~\ref{fig:std_PGL} satisfy these criteria. 
An \textcode{mmgroup} calculation shows that $|x_6|=6$ and $|x_3|=3$, and that each element $y_6x$ acts faithfully on $\langle g_{13} \rangle$, and one can check that $x_6,x_3 \in H$ by enumerating $H$. 
It remains to show that $x_6^2$ and $x_3$ are not conjugate in $H$. 
This and the final assertion can be verified in \textcode{mmgroup} as follows. 
In each of the five cases, \textcode{conjugate\_involution()} shows that $g_6^3 \in 2\mt{B}$ and yields an element $h \in \MM$ such that $(g_6^3)^h = z \in Z(\GG)$. 
In particular, $g_6^h \in \GG$ because $g_6$ commutes with $g_6^3$, so \textcode{chi\_G\_x0()} can be used to determine the $\MM$-class of $g_6$. 
It suffices to evaluate the character $\chi_\MM$ of the $196883$-dimensional $\mathbb{C}\MM$-module on each $g_6^h$, because $\chi_\MM$ distinguishes the $\MM$-classes of elements of order $6$. 
For $x=1$, we find that $\chi_\MM(g_6^h) = 77$, which indicates that $g_6 \in 6\mt{B}$. 
For $x \in \{x_6,x_6^2,x_6^3\}$, we have $\chi_\MM(g_6^h) = 5$, so $g_6 \in 6\mt{E}$. 
For $x=x_3$, we have $\chi_\MM(g_6^h) = -1$, so $g_6 \in 6\mt{F}$. 
In particular, $x_3$ is not conjugate to $x_6^2$ in $H$ because $y_6$ commutes with $H$. 
It remains to check that the groups $T_x = \langle g_{13}, y_6x \rangle \cong 13\mt{A}{:}6$ are not conjugate in $\MM$. 
If $(T_x)^h=T_{x'}$ with $x,x' \in \{1,x_6,x_6^2,x_6^3,x_3\}$ for some $h\in\MM$, then $h$ normalises $\langle g_{13} \rangle$, so $x=x'$ because $1$, $x_6$, $x_6^2$, $x_6^3$, and $x_3$ belong to distinct $H$-classes. 
\end{proof}

Next, we attempt to extend each of the five groups $\langle g_{13},g_6 \rangle \cong 13\mt{A}{:}6$ given in Proposition~\ref{lem_13:6} to a subgroup of $\MM$ isomorphic to $P = \PSL_2(13)$. 
Recall that we can assume that all involutions in our subgroups are of class $2\mt{B}$; see Section~\ref{sec_known_subgroups}. 
We first find the involutions $j_2 \in \MM$ such that $\langle g_6,j_2 \rangle$ is isomorphic to $\mt{D}_{12}$ and contains only $2\mt{B}$-involutions, and then test whether each $\langle g_{13},g_6,j_2 \rangle \cong P$. 

\begin{lemma} \label{lemma:classmult1}
If $h_6 \in \MM$ lies in class $6\mt{B}$, then there are exactly $14152320$ involutions $j_2 \in 2\mt{B}$ such that $\langle h_6,j_2 \rangle$ is isomorphic to $\mt{D}_{12}$ and contains only $2\mt{B}$-involutions. 
If $h_6 \in 6\mt{E}$, then there are exactly $466560$ such involutions. 
If $h_6 \in 6\mt{F}$, then there are exactly $91530$ such involutions. 
\end{lemma}

\begin{proof}
All elements of order $6$ in $\MM$ are conjugate to their inverses, so the required involutions are in one-to-one correspondence with the pairs of $2\mt{B}$-involutions whose product equals $h_6$. 
For each $\mt{X} \in \{\mt{B},\mt{E},\mt{F}\}$, the total number of such involutions is equal to the $(2\mt{B},2\mt{B},6\mt{X})$ class multiplication coefficient of $\MM$, which can be calculated in {\sf GAP} as demonstrated in Listing~\ref{figgap1}.
\end{proof}

For a given $g_6$, the involutions $j_2$ can all be found in $N_\MM(\langle g_6 \rangle)$. 
However, per the discussion in Section~\ref{sec_prel}, we have no easy way of constructing these groups, so we proceed as follows. 
Consider the projection $\pi : \GG \to {\rm Co}_1\leq \GL_{24}(2)$ described in Section \ref{secQ}. 
We know that $g_6^3 \in 2\mt{B}$, so $N_\MM(\langle g_6 \rangle) \leq C_\MM(g_6^3) \cong \GG$. 
We use \textcode{conjugate\_involution()} to find $h\in\MM$ such that $(g_6^3)^h=z\in Z(\GG)$, compute $g_6'=g_6^h\in\GG$, construct generators of the normaliser of $\langle \pi(g_6') \rangle$ in ${\rm Co}_1\leq \GL_{24}(2)$ in {\sc Magma}, lift these generators back to $\GG$, and adjust the preimages by elements of $\QQ$ to obtain coset representatives that normalise $\langle g_6' \rangle$. 
This gives us generators for a subgroup $Y'$ of $\GG$ that normalises $\langle g_6' \rangle$. 
The element $g_6'$ is also centralised by some elements of $\QQ$, namely those representing vectors that are fixed by $\pi(g_6')$ in the action of $\mt{Co}_1$ on $\QQ/Z(\QQ)\cong 2^{24}$. 
We compute a basis for the $1$-eigenspace of $\pi(g_6')$, take the corresponding elements of $\QQ$, and add these to our generating set for $Y'$. 
Conjugating by $h^{-1}$ gives us generators for a subgroup $Y$ of $N_\MM(\langle g_6 \rangle)$. 
We then find the involutions $j_2$ in $Y$ via random search, as described in Section~\ref{ss:algo}. 
Lemma~\ref{lemma:classmult1} tells us when we have found all of them. 

\begin{remark} \label{remPSL2(13)}
For brevity, we do not attempt to justify whether $Y$ is equal to $N_\MM(\langle g_6 \rangle)$. 
This is not necessary, given that we are able to find all of the involutions $j_2$.  
Our files containing the $j_2$ are too large (up to 8GB) to upload to our GitHub repository \cite{ourfile}, but the code given there includes generators for each $Y$, so the reader wishing to reproduce our proof can recover the $j_2$ via random search in $Y$. 
\end{remark}

The next step is to check which of the involutions $j_2$ extend $\langle g_{13},g_6 \rangle$ to a group isomorphic to $P$. 
Up to isomorphism, $P$ is the unique group of order $1092$ that has subgroups isomorphic to $13{:}6$ and $\mt{D}_{12}$, so it suffices to enumerate each group $S=\langle g_{13},g_6,j_2 \rangle$ and check its order. 
This, however, is not particularly efficient, so in practice we used the following observation: if $S \cong P$, then the multi-set $\{|t_2g_{13}|: t_2 \text{ is a non-central involution in } \langle g_6,j_2\rangle \}$ is equal to $\{3,6,7,7,7,13\}$. 
That is, we only checked whether $|S|=|P|$ if $S$ first satisfied this criterion. 
We also increased the efficiency of our computations in other ways: we noted that if some $j_2$ fails to extend $\langle g_{13},g_6 \rangle$ to $P$, then so does every involution in $\langle g_6,j_2 \rangle$ and every $C_\MM(\langle g_{13},g_6 \rangle)$-conjugate of such an involution; we ran certain computations in parallel using Python's \textcode{multiprocessing} package; and, to speed up look-ups in lists, we stored \textcode{mmgroup} elements \textcode{g} as the hashable objects \textcode{tuple(g.reduce().as\_tuples())}, which depend only on the group element defined by \textcode{g} and not on the representation of that element as a word in the generators for $\MM$ used in \textcode{mmgroup}. 
The outcome of these computations is as follows.

In the cases $g_6\in\{y_6x_6^2,y_6x_6^3,y_6x_3\}$, there are no involutions $j_2$ such that $S = \langle g_{13},g_6,j_2 \rangle$ is isomorphic to $P = \PSL_2(13)$.
For $g_6=y_6$, there are exactly $312$ involutions $j_2$ such that $S \cong P$. 
We claim that the $312/6 = 52$ distinct groups $S$ are all conjugate in $\MM$ and have centralisers of order $108$. 
By construction, the centraliser $C = C_\MM(T)$ of $T = \langle g_{13},y_6 \rangle<S$ is the group $H \cong \PSL_3(3)$ generated by the elements $c$ and $d$ in Listing~\ref{fig:std_PGL}. 
A direct calculation in \textcode{mmgroup} shows that exactly $108$ elements of $H$ also centralise $j_2$ in each case, so $|C_\MM(S)|=108$.
If some element of $C$ normalises $S$, then it must centralise $S$, because the only automorphism of $P$ that centralises a maximal $13{:}6$ is the identity. 
The stabiliser of $S$ under conjugation by $C$ is therefore contained in $C_\MM(S)$, so the length of the $C$-orbit containing $S$ is at least, and hence exactly, $|C|/|C_\MM(S)| = 5616/108 = 52$. 
Finally, in the case $g_6=y_6x_6$, there are exactly $48$ involutions $j_2$ such that $S \cong P$, but the groups $S$ fall into two $\MM$-classes. 
Because $g_6 = y_6x_6$ and $x_6 \in H$ commutes with $y_6$, the group $C = C_\MM(T)$ is equal to $C_H(x_6) = \langle x_6 \rangle \cong 6$, so $|C_\MM(S)|$ is equal to the number of distinct powers of $x_6$ that centralise $j_2$. 
An \textcode{mmgroup} calculation shows that $12$ of the $48$ involutions $j_2$ yield $|C_\MM(S)|=3$, and the remaining $36$ yield $|C_\MM(S)|=1$. 
The same argument as above shows that the two distinct groups $S$ in the first case are conjugate in $\MM$, and that the six distinct groups $S$ in the second case are conjugate in $\MM$.

The groups $S \cong P = \PSL_2(13)$ with non-trivial centralisers comprise the two known classes of subgroups of $\MM$ isomorphic to $P$ discussed in Section~\ref{sec_known_subgroups}. 
We now show that the groups $S$ with trivial centralisers extend to maximal subgroups of $\MM$ isomorphic to $\text{Aut}(\PSL_2(13)) = \PGL_2(13) = \PSL_2(13){:}2$. 
Given that our six groups of this form are all conjugate in $\MM$, it suffices to consider one of them. 
One such group is generated by the elements $g_{13}$, $g_6$, and $i_2$ given in Listing~\ref{fig:std_PGL213final}, where $g_{13}$ is the same as in Listing~\ref{fig:std_PGL}, $g_6 = y_6x_6$ with $y_6$ and $x_6$ as in Listing~\ref{fig:std_PGL}, and $i_2$ is one of the involutions $j_2$ inverting $g_6$.  
The following result shows that we can extend this group $S$ as claimed by adjoining a certain element of order $12$ found in $N_\MM(\langle g_{13} \rangle)$. 
This completes the proof of Theorem~\ref{thmL213}.
 
\begin{proposition}\label{thm_PGL}
The subgroup of $\MM$ generated by the elements $g_{13}$, $g_6$, $i_2$, and $a_{12}$ defined in \textcode{mmgroup} format in Listing~\ref{fig:std_PGL213final} is a maximal subgroup of $\MM$ isomorphic to $\PSL_2(13){:}2$. 
It has trivial centraliser in~$\MM$, and non-trivial intersection with precisely the $\MM$-classes $2\mt{B}$, $3\mt{B}$, $4\mt{C}$, $6\mt{E}$, $7\mt{B}$, $12\mt{H}$, $13\mt{A}$, and $14\mt{C}$. 
\end{proposition}

\begin{proof}
Let $G = \langle g_{13},g_6,i_2,a_{12} \rangle$. 
An \textcode{mmgroup} calculation shows that $u = g_{13}^6$ and $v=i_2$ satisfy the presentation $\langle u,v\mid u^{13}=v^2=(uv)^3=(u^2vu^7v)^3=1\rangle$ for $\PSL_2(13)$ given in \cite[Theorem~A]{BM}. 
Because $u$ and $v$ are non-trivial,  $\langle u,v \rangle \cong \mathrm{PSL}_2(13)$ by Von Dyck's Theorem \cite[Theorem~2.53]{handbook}. 
A further \textcode{mmgroup} calculation shows that $g_6, a_{12}^2 \in \langle u,v \rangle$, and that $a_{12}$ normalises $\langle u, v \rangle$. 
We know from the preceding discussion that $T = \langle g_{13},g_6\rangle < G$ has the form $13{:}6$ and that its centraliser in $\MM$ is equal to $\langle x_6 \rangle$, where $x_6$ is defined in Listing~\ref{fig:std_PGL}. 
An \textcode{mmgroup} calculation confirms that $i_2$ is one of the involutions $j_2$ discussed in the corresponding case above; that is, $i_2$ belongs to the $\MM$-class $2B$, inverts $g_6$, extends $T$ to a group $S$ isomorphic to $\PSL_2(13)$, and commutes with only the identity element of $\langle x_6 \rangle = C_\MM(T)$, so that $C_\MM(S)=1$. 
Because there are no elements of order 12 in $\mathrm{PSL}_2(13)$, it follows that $G \cong \PSL_2(13){:}2$.

Next, we justify the claims about conjugacy class fusion. 
The calculations described here are given in detail in our supporting code \cite{ourfile}. 
For the sake of exposition, we label $G$-classes by lowercase letters and $\MM$-classes by uppercase letters. 
The group $G$ has exactly $15$ conjugacy classes, labelled $1\mt{a}$, $2\mt{a}$--$\mt{b}$, $3\mt{a}$, $4\mt{a}$, $6\mt{a}$, $7\mt{a}$--$\mt{c}$, $12\mt{a}$--$\mt{b}$, $13\mt{a}$, and $14\mt{a}$--$\mt{c}$ in its character table in {\sf GAP}, constructed as {\sf CharacterTable("L2(13).2")}. 
We know from Propositions~\ref{prop_13a} and \ref{lem_13:6} that $6\mt{a}$ and $13\mt{a}$ fuse to $6\mt{E}$ and $13\mt{A}$, and the power maps in the character table of $\MM$ show that $6\mt{E}$-elements power to $3\mt{B}$-elements. 
The involutions in the socle of $G$ lie in $2\mt{b}$; one of them is $i_2$, so $2\mt{b}$ fuses to $2\mt{B}$. 
All elements of order $7$ in $G$ lie in a single $\MM$-class because the $\MM$-classes of elements of order $7$ are rational; cf. the argument at the beginning of Section~\ref{sec_PSLPGL}. 
The element $g_{14} = a_{12}i_2g_{13}^2$ has order $14$ and commutes with $g_{14}^7 \in 2\mt{B}$. 
We can therefore conjugate $g_7 = g_{14}^2$ into $\GG$ and check that $\chi_\MM(g_7)=1$. 
This shows that $7\mt{a}$--$\mt{c}$ fuse to $7\mt{B}$. 
(Note that this is consistent with the fusion restriction for order-$7$ elements mentioned in Section~\ref{sec_known_subgroups}.) 
The only elements of order $14$ in $\MM$ that power to $7\mt{B}$-elements are those in class $14\mt{C}$, so $14\mt{a}$--$\mt{c}$ all fuse to $14\mt{C}$. 
On the other hand, elements in $14\mt{a}$--$\mt{c}$ power to $2\mt{a}$, and $14\mt{C}$-elements power to $2\mt{B}$, so $2\mt{a}$ also fuses to $2\mt{B}$. 
All elements of order $12$ in $G$ lie in a single $\MM$-class because $G$ has a unique class of cyclic subgroups of order $12$ and all $\MM$-classes of elements of order $12$ are rational. 
It therefore suffices to show that $a_{12} \in 12\mt{H}$. 
Given that $a_{12}^6 \in 2\mt{B}$, this can be done by conjugating $a_{12}$ into $\GG$ and checking that its $\chi_\MM$-value is $13$. 
Finally, $12\mt{H}$-elements power to $4\mt{C}$-elements. 

It remains to establish that $G$ is maximal in $\MM$. 
We do this by showing that $G$ is not contained in any other maximal subgroup of $\MM$. 
Per Section~\ref{sec_intro}, the maximal subgroups of $\MM$ are classified subject to the condition that any further maximal subgroup is almost simple with socle $\PSU_3(4)$ or $\PSL_2(q)$, $q \in \{8,13,16\}$. 
The order of $G$ does not divide the order of the automorphism group of $\PSU_3(4)$, $\PSL_2(8)$, or $\PSL_2(16)$, so we just need to show that $G$ is not contained in any of the known maximal subgroups of $\MM$, which are listed in Table~\ref{tab:allmax} (bearing in mind Remark~\ref{remPSL259}b). 
Most known maximal subgroups $L$ of $\MM$ cannot contain $G$ either because $|G|$ does not divide $|L|$ (which is usually coprime to $13$), or because $L$ has non-trivial centre while $C_\MM(G)=1$. 
The remaining known maximal subgroups are $2^{2} \udot ^{2}\mt{E}_6(2){:}\mt{S}_3$, $3^{1+12}\udot 2\udot \mt{Suz}{:}2$, $\mt{S}_3\times\mt{Th}$, $13^2{:}2\mt{PSL}_2(13).4$, $3^8\udot\mt{P}\Omega_8^-(3).2$, and $(3^2{:}2\times\mt{P}\Omega^+_8(3))\udot\mt{S}_4$. 
(Note that, in the Atlas, the fourth of these groups is written as $13^2{:}4\mt{PSL}_2(13)\udot 2$, and the notation ``$\mt{O}$'' is used instead of ``$\mt{P}\Omega$''.)
In the first four cases, it is known how $L$-classes fuse in $\MM$, and one can look up this information in {\sf GAP} as demonstrated in Listing~\ref{figgap1}. 
One finds that $L$ does not intersect the $\MM$-class $7\mt{B}$, $13\mt{A}$, $7\mt{B}$, and $3\mt{B}$, in the four respective cases. 
The class fusion from $G$ to $\MM$ calculated above therefore shows that $G$ is not contained in any conjugate of $L$ in these cases. 
Now consider $L = (3^2{:}2\times\mt{P}\Omega^+_8(3))\udot\mt{S}_4$. 
We claim that $L$ does not intersect the $\MM$-class $7\mt{B}$ (whereas $G$ does). 
The group $L$ can be constructed in {\sf GAP} as {\sf AtlasGroup("(3\^{}2:2xO8+(3)).S4")}. 
It has a single class of elements of order $7$, with centraliser of order $2^6{\cdot}3^3{\cdot}7$. 
The order of the centraliser of a $7\mt{B}$-element in $\MM$ has $3$-part only $3^2$, so the order-$7$ elements in $L$ fuse to $\MM$-class $7\mt{A}$, as claimed. 
Finally, consider $L = 3^8\udot\mt{P}\Omega_8^-(3).2$. 
The online Atlas~\cite{atlas-web} provides generators for a permutation representation of $L$ of degree $805896$. 
Upon constructing this permutation representation in {\sc Magma}, one finds that $L$ contains a unique class of subgroups isomorphic to $\PSL_2(13)$. 
However, each such subgroup is centralised by an element of order $3$ in $L$, so it is not conjugate in $\MM$ to the group $G$, which has trivial centraliser. 
\end{proof}  

\begin{remark} \label{rem:N13}
We explain how we found the $13\mt{A}$-element $g_{13}$ and the index-$2$ subgroup of $N_\MM(\langle g_{13} \rangle)$ given in Proposition~\ref{prop_13a}. 
Note that these details have no bearing on the correctness of the proof of Theorem~\ref{thmL213}; they are provided in case they are of interest. 
Recall the discussion about computing in \textcode{mmgroup} from Sections~\ref{sec2.3}--\ref{ss:algo}. 
It is easy to find a $13\mt{A}$-element by random selection in $\MM$, e.g. every element of order $104$ powers to $13\mt{A}$. 
Constructing the normaliser is far more difficult. 
If we could find a $13\mt{A}$-element in $\GG$, then we could construct at least part of its normaliser by passing to $\mt{Co}_1$ in {\sc Magma}. 
However, per Listing~\ref{figgap1}, all elements of order $13$ in $\GG$ are of class $13\mt{B}$. 
The normalisers in $\MM$ of $13\mt{A}$- and $13\mt{B}$-elements are $N_\MM(13\mt{A}) = ((13{:}6)\times \PSL_3(3)).2$ and $N_\MM(13\mt{B}) = 13^{1+2}{:}(3\times 4\udot \mt{S}_4)$. 
We first constructed $N_\MM(13\mt{B})$, the idea being to find a $13\mt{A}$-element in $N_\MM(13\mt{B})$ and build its normaliser beginning in $N_\MM(13\mt{B})$. 
Wilson \cite[Sections~4--6]{W14} describes how to construct $N_\MM(13\mt{B})$ from standard generators for $\GG$. 
We had found such generators, so we were able to produce a copy $U$ of $N_\MM(13\mt{B})$. 
This group is called {\sf "13\^{}(1+2):(3x4S4)"} in the {\sf GAP} character table library, and known class fusion data shows that only one of its classes of order-$13$ elements fuses to $13\mt{A}$ in $\MM$.  
We found $g_{13}$ by random search in $U$, and confirmed that it was of the desired class by calculating part of its centraliser in $U$. 
We then constructed the index-$2$ subgroup $K \cong 13^2{:}(6 \times 4)$ of $N_U(\langle g_{13} \rangle) \cong 13^2{:}(12\times 4)$ by random search in $U$, with the aim of extending $K$ to the index-$2$ subgroup $N \cong (13{:}6)\times \PSL_3(3)$ of $N_\MM(\langle g_{13} \rangle)$. 
The group $N$ has three classes of involutions, one of which has centraliser $6 \times \PSL_3(3)$. 
If $h_2 \in N$ is such an involution, then $N = \langle g_{13},C_N(h_2) \rangle$. 
Per the discussion in the proof of Proposition~\ref{lem_13:6}, we can obtain $h_2$ as the cube of some $h_6 \in 6\mt{B}$. 
We found the latter by random search in $K$. 
A $6\mt{B}$-element powers to a $2\mt{B}$-involution, so we could confirm that $h_6 \in 6\mt{B}$ as described in the proof of Proposition~\ref{lem_13:6}, and conjugate $C_K(h_2) < K$ into $\GG$. 
This enabled us to pass to $\mt{Co}_1$ in {\sc Magma} to obtain further elements commuting with $h_2$ by following the general procedure described in Section~\ref{secQ}. 
Eventually, we were able to construct the whole of $C_N(h_2)$ and thereby obtain our generators for $N$. 
\end{remark}

\begin{remark}\label{rem_no13B} 
Wilson proved that $\MM$ has no subgroups isomorphic to $\PSL_2(13)$ containing $13\mt{B}$-elements by showing that the order-$6$ elements of every $13\mt{B}{:}6<\MM$ lie in $6\mt{F}$, and that there is no involution that extends such a subgroup to a group isomorphic to $\PSL_2(13)$. 
We reproduced his proof using the methodology described in this section, and can confirm that we reached the same conclusion. 
\end{remark}

\section{Intermezzo --- Subgroups of $\MM$ isomorphic to $\mt{A}_5$}\label{secA5}

\noindent Both of the groups $\PSL_2(16)$ and $\PSU_3(4)$ contain subgroups isomorphic to $\mt{A}_5$. 
The former contains $\mt{A}_5$ as a maximal subgroup, and the latter has a maximal subgroup $5 \times \mt{A}_5$. 
We can therefore attempt to generate subgroups $\PSL_2(16)$ or $\PSU_3(4)$ of $\MM$ by starting with (appropriate) subgroups  $\mt{A}_5$, which have been classified by Norton~\cite[Section~4]{N98}. 
We collect some information about those conjugacy classes of $\mt{A}_5 < \MM$ that could, in principle, lead to `new' subgroups $\PSL_2(16)$ or $\PSU_3(4)$. 
Recall that $\mt{A}_5$ has unique conjugacy classes of elements of orders $2$ and $3$, and two classes of elements of order~$5$. 
An argument similar to the one given in the first paragraph of Section~\ref{sec_PSLPGL} shows that all elements of order $5$ in a subgroup $\mt{A}_5$ of $\MM$ must belong to a single $\MM$-class of elements of order $5$. 

Per Section~\ref{sec_known_subgroups}, in every as-yet-unclassified subgroup of $\MM$ isomorphic to $\PSL_2(16)$ or $\PSU_3(4)$, the elements of order $5$ lie in the $\MM$-class $5\mt{B}$. 
By \cite[Table 3]{N98}, there are eight conjugacy classes of $\mt{A}_5 < \MM$, but only three contain $5\mt{B}$-elements. 
Table~\ref{tab_A5} shows the $\MM$-classes containing the elements of orders $2$, $3$, and $5$ in each such $\mt{A}_5$. 
The two classes of $\mt{A}_5$ that intersect $2\mt{B}$, $3\mt{B}$, and $5\mt{B}$ can be distinguished by their centralisers as indicated in Table~\ref{tab_A5}. 
Note that the centraliser of each such $\mt{A}_5$ contains only $2\mt{A}$-involutions, so it is not possible to find such an $\mt{A}_5$ in any conjugate of the maximal subgroup $\GG$ of $\MM$. 
Per Norton~\cite{N98} and Holmes and Wilson~\cite{HW08}, these two $\mt{A}_5 < \MM$ are said to be of {\em type} $\mt{T}$ and $\mt{B}$, respectively. 
The third $\mt{A}_5 < \MM$ in Table~\ref{tab_A5} intersects $2\mt{B}$, $3\mt{C}$, and $5\mt{B}$, so Holmes and Wilson~\cite{HW08} say that it has type $\mt{BCB}$; we shall say that it has type $\mt{G}$ because we find a copy in $\GG$. 

\renewcommand{\arraystretch}{1.12}
\begin{table}[!t]
\begin{tabular}{l|cccccc}
\hline
Type & $A_\mt{G}< \GG$ &\;& $A_\mt{T}< \mt{Th} < (\mt{S}_3 \times \mt{Th}) \cap 2\udot\BB$ &\;& $A_\mt{B}< 2\udot\BB$\\
\hline
Class fusion in $\MM$ & $(2\mt{B},3\mt{C},5\mt{B})$ && $(2\mt{B},3\mt{B},5\mt{B})$ && $(2\mt{B},3\mt{B},5\mt{B})$ \\
$C_\MM(\mt{A}_5)$ & $\mt{D}_{10}$ && $\mt{S}_3$ && $2$ \\ 
$C_\MM(C_\MM(\mt{A}_5))$ & $5^3 \udot (4 \times \mt{A}_5)$ && $\mt{Th}$ && $2 \udot \BB$ \\ \hline
\end{tabular}\\
\caption{The conjugacy classes of $\mt{A}_5 < \MM$ containing $5\mt{B}$-elements; see \cite[Table~3]{N98}. 
}\label{tab_A5}
\end{table}

Let us generically denote by $A_\mt{G}$, $A_\mt{T}$, or $A_\mt{B}$, respectively, a subgroup $\mt{A}_5$ of $\MM$ of type $\mt{G}$, $\mt{T}$, or $\mt{B}$. 
We reiterate that each such $\mt{A}_5$ is unique up to conjugacy. 
We were able to find a subgroup $A_\mt{G}$ in $\GG$ via random search; generators for such a subgroup are given in Proposition~\ref{prop_A5}. 

It was also relatively straightforward to find a subgroup $A_\mt{T}$ of $\MM$. 
The construction is summarised in Proposition~\ref{prop_A5}, but we first describe the basic strategy, which resembles that of Remark~\ref{rem:N13}. 
We first used Bray's method \cite{bray} to find various elements in the centraliser $2\udot \BB < \MM$ of a certain $2\mt{A}$-involution. 
We were able to deduce that, amongst these elements, we had a pair of generators $a$ and $b$ for $2\udot \BB$ that project to standard generators $c$ and $d$ for $\BB$ itself, in the sense defined by Wilson~\cite{W96} and the online Atlas~\cite{atlas}; namely, generators $c,d\in\BB$ such that $c$ belongs to $\BB$-class $2\mt{C}$, $d$ belongs to $\BB$-class $3\mt{A}$, $cd$ has order $55$, and $(cd)^4(dc)^2d^2cd^2$ has order $23$. 
The online Atlas provides SLPs for constructing various subgroups of $\BB$ from standard generators. 
We were thereby able to construct all subgroups in the chain $\mt{S}_5 < \mt{Th} < 2\udot \BB$, and then we finally found a copy of $A_\mt{T}$ in the $\mt{S}_5$. 
Proposition~\ref{prop_A5} shows that our copy of $A_\mt{T}$ does indeed have type $\mt{T}$. 
(Our generators for $2\udot \BB$ and the maximal subgroup $\mt{S}_3 \times \mt{Th}$ of $\MM$, i.e. the normaliser of a $3\mt{C}$-element, are provided in our GitHub repository \cite{ourfile}.) 

Constructing a copy of $A_\mt{B}$ was significantly more difficult. 
We first constructed a subgroup $\mt{A}_5\times \mt{A}_{12}$ of $\MM$, which has index $2$ in a maximal subgroup $(\mt{A}_5\times \mt{A}_{12}){:}2$. 
As explained by Norton~\cite[Section~4]{N98}, the group $\mt{A}_5\times \mt{A}_{12}$ contains subgroups $\mt{A}_5$ of both types $\mt{T}$ and $\mt{B}$ as diagonal subgroups, and it is possible to distinguish between the two types by considering orbit lengths in the natural $12$-point permutation representation of the $\mt{A}_{12}$. 
This allowed us to find a copy of $A_\mt{B}$ in $\mt{A}_5\times \mt{A}_{12}$. 
Proposition~\ref{prop_A12} provides generators for our copy of $\mt{A}_5\times \mt{A}_{12} < \MM$. 
Proposition~\ref{prop_A5} shows that our $A_\mt{B} < \mt{A}_5\times \mt{A}_{12}$ does indeed have type $\mt{B}$. 
Remark~\ref{remA12} explains how we constructed the $\mt{A}_5\times \mt{A}_{12}$ in the first place. {We note that \cite[Lemma 6.1]{pispop} provides an alternative proof that this subgroup does  indeed have type $\mt{B}$.}

\begin{proposition}\label{prop_A12}
The elements $x_3$ and $x_{10}$ defined in \textcode{mmgroup} format in Listing~\ref{fig:A5A12} generate a subgroup of $\MM$ isomorphic to $\mt{A}_{12}$. 
The elements $a_2$ and $a_3$ defined in Listing~\ref{fig:A5A12} generate a subgroup of $\MM$ isomorphic to $\mt{A}_5$ that commutes with $\langle x_3, x_{10} \rangle$. 
In particular, $\langle x_3,x_{10},a_2,a_3 \rangle \cong \mt{A}_5 \times \mt{A}_{12} < \MM$.
\end{proposition}

\begin{proof}
A direct calculation in \textcode{mmgroup} shows that $x_3$ and $x_{10}$ satisfy the presentation
\begin{align*}
\langle x_3, x_{10} \mid 
&x_3^3 = x_{10}^{10} = (x_3x_{10})^{11} = [x_3,x_{10}]^2 = \\
&(x_3x_{10}^{-2}x_3x_{10}^2)^2 = [x_3,x_{10}^3]^2 = (x_3x_{10}^{-4}x_3x_{10}^4)^2 = [x_3,x_{10}^5]^2 = 1 \rangle
\end{align*}
for $\mt{A}_{12}$; see \cite[p.~67]{CoxeterMoser}. 
Similarly, $a_2$ and $a_3$ satisfy the presentation $\langle a_2, a_3 \mid a_2^2=a_3^3=(a_2a_3)^5=1 \rangle$ for $\mt{A}_5$, and $a_2$ and $a_3$ commute with $x_3$ and $x_{10}$. 
Given that both $\mt{A}_{12}$ and $\mt{A}_5$ are simple groups, the result follows from Von Dyck's Theorem \cite[Theorem~2.53]{handbook}. 
\end{proof}

The following result summarises our constructions of the $\mt{G}$-, $\mt{T}$-, and $\mt{B}$-type subgroups $\mt{A}_5$ of $\MM$.

\begin{proposition}\label{prop_A5}
The subgroups of $\MM$ generated by the elements $g_2$ and $g_3$ given in Listing~\ref{fig:A5} are all isomorphic to $\mt{A}_5$ and have types $\mt{G}$, $\mt{T}$, and $\mt{B}$, according to the ``type'' indicated in the listing. 
\end{proposition}

\begin{proof}
In each case, an \textcode{mmgroup} calculation confirms that $g_2$ and $g_3$ satisfy the presentation for $\mt{A}_5$ given in the proof of Proposition~\ref{prop_A12}, namely, $g_2^2=g_3^3=(g_2g_3)^5=1$. 
It remains to verify that $\langle g_2,g_3 \rangle$ has type $\mt{G}$, $\mt{T}$, and $\mt{B}$ in the three respective cases. 
Note that the arguments that follow involve the auxiliary elements $c_i$, $i_2$, and $h$ defined in Listing~\ref{fig:A5}. 
Write $g_5 = g_2g_3$ in each case. 

In the first case (type $\mt{G}$), $c_2$ is the central involution $z$ in the fixed copy of the maximal subgroup $\GG \cong \Gx$ of $\MM$ in \textcode{mmgroup}; see Section~\ref{sec_MM}. 
A direct calculation shows that $c_2$ centralises $g_2$ and $g_3$, i.e. $\langle g_2,g_3 \rangle < \GG$. 
The method \textcode{conjugate\_involution()} shows that $g_2 \in 2\mt{B}$, and \textcode{chi\_G\_x0()} shows that the character values of $g_3$ and $g_5$ in the $196883$-dimensional complex representation of $\MM$ are $-1$ and $8$, so $g_3 \in 3\mt{C}$ and $g_5 \in 5\mt{B}$. 
It therefore follows from Table~\ref{tab_A5} that $\langle g_2,g_3 \rangle$ has type $\mt{G}$. 

In the second case (type $\mt{T}$), \textcode{conjugate\_involution()} confirms that $g_2 \in 2\mt{B}$. 
The element $i_2$ is a $2\mt{B}$-involution that centralises $g_5$, but not $g_2$ or $g_3$. 
The element $h$ conjugates $i_2$ to $z$, and therefore conjugates $g_5$ into $\GG$. 
The method \textcode{chi\_G\_x0()} shows that $g_5 \in 5\mt{B}$, so it follows from Table~\ref{tab_A5} that $\langle g_2,g_3 \rangle$ has type $\mt{G}$, $\mt{T}$, or $\mt{B}$. 
The elements $c_2$ and $c_3$ centralise $\langle g_2,g_3 \rangle$, and have orders $2$ and $3$, respectively. 
In particular, $\langle g_2,g_3 \rangle$ is centralised by an element of order $3$, so it must be of type $\mt{T}$. 

In the third case (type $\mt{B}$), $g_2 \in 2\mt{B}$, and $i_2$ and $h$ have the same properties as in the second case, so proceeding as in that case confirms that $g_5 \in 5\mt{B}$, whence $\langle g_2,g_3 \rangle$ has type $\mt{G}$, $\mt{T}$, or $\mt{B}$. 
The element $c_2$ centralises $\langle g_2,g_3 \rangle$ and lies in the $\MM$-class $2\mt{A}$, so $\langle g_2,g_3 \rangle$ is contained in a conjugate of $2 \udot \BB$. 
This shows that $\langle g_2,g_3 \rangle$ does not have type $\mt{G}$, because $A_\mt{G}$ is contained in a $2\mt{B}$-centraliser and $C_\MM(A_\mt{G}) \cong \mt{D}_{10}$ has a unique class of involutions. 
To show that $\langle g_2,g_3 \rangle$ has type $\mt{B}$, we explain how we constructed it as a diagonal subgroup of the copy of $\mt{A}_5 \times \mt{A}_{12}$ given in Proposition~\ref{prop_A12}. 
The proof of Proposition~\ref{prop_A12} says (in particular) that the elements $x_3$ and $x_{10}$ defined in Listing~\ref{fig:A5A12} satisfy a certain presentation for $\mt{A}_{12}$. 
The elements $y_3=(1,2,3)$ and $y_{10} = (1,3)(2,4,5,6,7,8,9,10,11,12)$ of the symmetric group on $\Omega = \{1,\ldots,12\}$  satisfy the same presentation. 
We constructed $\langle y_3,y_{10} \rangle \cong \mt{A}_{12}$ in {\sc Magma} \cite{magma} and found a subgroup $A \cong \mt{A}_5$ such that a diagonal subgroup of $\mt{A}_5 \times A < \mt{A}_5 \times \mt{A}_{12}$ has type $\mt{B}$. 
This was done by considering the orbits of $A$ on $\Omega$. 
According to \cite[Table~4, rows~5--7]{N98}, $A$ should have orbit lengths $12$; $6$ and $6$; or $6$, $5$, and~$1$. 
We found $A \cong \mt{A}_5$ with orbit lengths $6$ and $6$, and used the {\sc Magma} function {\sc InverseWordMap} to record generators for $A$ as words in $y_3$ and $y_{10}$. 
This allowed us to use our generators $x_3$ and $x_{10}$ for $\text{A}_{12}$ in \textcode{mmgroup} to construct a subgroup $\overline{A}$ of $\langle x_3,x_{10} \rangle$ that is an image of $A$ under some automorphism of $\mt{A}_{12}$. 
(For brevity, we do not include the SLPs here; see \cite{ourfile} instead.)
Because all automorphisms of $\mt{A}_{12}$ preserve cycle structure in the natural $12$-point representation, the group $\overline{A} < \langle x_3,x_{10} \rangle$ also has the required orbit-length property. 
The group $\overline{A}$ is generated by the elements $b_2$ and $b_3$ in Listing~\ref{fig:A5A12}, which satisfy the aforementioned presentation for $\mt{A}_5$, i.e. $b_2^2=b_3^3=(b_2b_3)^5=1$. 
The elements $g_2$ and $g_3$ given under ``type $\mt{B}$'' in Listing~\ref{fig:A5} generate a diagonal subgroup of $\mt{A}_5 \times \overline{A} < \mt{A}_5 \times \mt{A}_{12} < \MM$. 
\end{proof}

\begin{remark}
Note that neither of the groups $\langle a_2,a_3 \rangle \cong \mt{A}_5$ nor $\langle b_2,b_3 \rangle \cong \mt{A}_5$ given in Listing~\ref{fig:A5A12} has type $\mt{G}$, $\mt{T}$, or $\mt{B}$. 
Per \cite[Tables~3--4]{N98}, $\langle a_2,a_3 \rangle$ belongs to the unique class of $\mt{A}_5 < \MM$ intersecting $2\mt{A}$, $3\mt{A}$, and $5\mt{A}$; and $\langle b_2,b_3 \rangle$ belongs to the unique class intersecting $2\mt{B}$, $3\mt{A}$, and $5\mt{A}$. 
\end{remark}

\begin{remark}\label{remA12}
Although they are not needed for the proof of Proposition~\ref{prop_A12}, we provide some details as to how we constructed the subgroup $\mt{A}_5\times\mt{A}_{12}$ of $\MM$ given in Listing~\ref{fig:A5A12}. 
Recall that we had found generators for a copy of $2 \udot \BB$ in $\MM$ that project to standard generators for $\BB$, enabling us to use SLPs from the online Atlas to construct various subgroups of $\BB$ or $2 \udot \BB$. 
To construct $\mt{A}_5\times\mt{A}_{12}$, we considered the chain of subgroups $\mt{A}_{12} < \mt{HN}<\mt{HN}{:}2<\BB$. 
We constructed $2 \times \text{A}_{12}<2\udot \BB$, and then used a random search to find the elements $x_3$ and $x_{10}$ given in Listing~\ref{fig:A5A12}, which generate a copy $X$ of $\text{A}_{12}$. 
To find the group $\langle a_2,a_3 \rangle \cong \mt{A}_5$ given in Listing~\ref{fig:A5A12}, which centralises $X \cong \mt{A}_{12}$, we first found a $2\mt{B}$-involution in $X$ and conjugated it to the central involution in $\GG$ via some $h \in \MM$. 
We used {\sc Magma} to construct the centraliser of $\pi(X^h \cap \GG)$ in $\pi(\GG) \cong \mt{Co}_1$, where $\pi \colon \GG\to \GL_{24}(2)$ is the homomorphism described in Section~\ref{secQ}, and pulled back generators of this centraliser to $\GG$. 
We then adjusted the pulled-back generators by elements of the normal subgroup $\QQ$ of $\GG$ to obtain coset representatives centralising $X^h \cap \GG$. 
Finally, we used a random search in the group generated by these elements to find sufficiently many elements that commute with the whole of $X^h$ and generate a group isomorphic to $\mt{A}_5$. 
\end{remark}

\section{Proof of Theorem~\ref{thmL216}}\label{sec_L16}

\noindent The general strategy of our proof of Theorem~\ref{thmL213} applies also to the proofs of Theorems~\ref{thmU34}--\ref{thmL28}, but there are various key differences in each case. 
We aim to provide details on the latter, while keeping the general discussion more concise than in Section~\ref{sec_PSLPGL}. 
We first deal with Theorem~\ref{thmL216}. 

The group $P = \PSL_2(16)$ has order $2^4{\cdot}3{\cdot}5{\cdot}17$ and maximal subgroups $2^4{:}15$, $\mt{A}_5$, $\mt{D}_{34}$, and $\mt{D}_{30}$, all unique up to conjugacy. 
There are two classes of elements of order $5$ in $P$, both of which intersect a maximal $\mt{A}_5$ non-trivially. 
For each element $g_5$ of order $5$ in a fixed maximal $\mt{A}_5 < P$, there are exactly $10$ involutions $j_2\in P$ such that $\langle g_5,j_2\rangle\cong \mt{D}_{10}$ and $\langle \mt{A}_5,j_2\rangle=P$. 
As explained in Sections~\ref{sec_known_subgroups} and~\ref{secA5}, every as-yet-unclassified subgroup of $\MM$ isomorphic to $P$ must have its elements of orders $2$ and $5$ lying in the $\MM$-classes $2\mt{B}$ and $5\mt{B}$. 
In particular, a maximal $\mt{A}_5$ in such a subgroup must have type $\mt{G}$, $\mt{T}$, or $\mt{B}$. 
Given that $A_\mt{G}$, $A_\mt{T}$, and $A_\mt{B}$ are unique up to conjugacy in $\MM$, it suffices to classify the subgroups of $\MM$ isomorphic to $P$ that contain one of the {\em fixed} groups $A_\mt{G}$, $A_\mt{T}$, or $A_\mt{B}$ defined in Listing~\ref{fig:A5}.

Let $A = \langle g_2,g_3 \rangle$ be one of the groups $A_\mt{G}$, $A_\mt{T}$, or $A_\mt{B}$ in Listing~\ref{fig:A5}, and note that we also refer to some of the other elements defined there.
Recall from the proof of Proposition~\ref{prop_A5} that $g_5 = g_2g_3$ has order $5$ in each case, and that, in the second and third cases, $i_2$ is a $2\mt{B}$-involution centralising $g_5$. 
Let us define $i_2$ to be the central involution in $\GG$ when $A=A_\mt{G}$, so that we can discuss all three cases at the same time.
We need to find all involutions $j_2 \in \MM$ that invert $g_5$ by conjugation, because such involutions are precisely those yielding $\langle g_5,j_2 \rangle \cong \mt{D}_{10}$. 
All such involutions lie in the normaliser $N=N_\MM(\langle g_5\rangle)$ of $\langle g_5 \rangle$ in $\MM$, so we need to find a subgroup of $N$ that contains all of them. 
Because $g_5 \in 5\mt{B}$, it follows from \cite[Theorem 5]{W88} that $N \cong 5^{1+6}{:}2\udot\mt{J}_2{:}4$, where $\mt{J}_2$ is the second Janko group. 
Recall that $A_\mt{G} < \GG$, and that, in the other two cases, $h$ conjugates $\langle g_5,i_2 \rangle$ into $\GG$. 
Let us set $h=1$ in the case $A = A_\mt{G}$. 

The group $\GG \cong \Gx$ intersects the $\MM$-class $5\mt{B}$ in two $\GG$-classes. 
These $\GG$-classes are labelled $5\mt{A}$ and $5\mt{C}$ in the character table of $\GG$ in {\sf GAP}~\cite{GAPbc,gap}. 
They project to classes also labelled $5\mt{A}$ and $5\mt{C}$ (respectively) in the character table of the quotient $\mt{Co}_1$, and can be distinguished by the dimension of their fixed-point spaces on the $24$-dimensional module for $\mt{Co}_1$ in characteristic $2$. 
Specifically, the $5\mt{A}$-elements have a trivial fixed-point space, and the $5\mt{C}$-elements have a $4$-dimensional fixed-point space. 
It turns out that if $A=A_\mt{G}$ or $A_\mt{B}$ then $g_5^h$ belongs to the $\GG$-class $5\mt{C}$, and if $A=A_\mt{T}$ then $g_5^h$ belongs to the $\GG$-class $5\mt{A}$. 
This can be verified using the \textcode{mmgroup} method \textcode{chi\_G\_x0()}, specifically from the values of the character $\chi_{24}$ defined in Listing~\ref{fig:mmgroup1}; see \cite{ourfile}.

To find a sufficiently large subgroup of $N$, we proceed as follows. 
The element $h$ gives us a conjugate $\langle g_5^h \rangle$ of the subgroup $\langle g_5 \rangle$ of $N$ inside $\GG$. 
(Recall that $h=1$ when $A=A_\mt{G}$.) 
We use the homomorphism $\pi \colon \GG\to \GL_{24}(2)$ described in Section~\ref{secQ} to construct the normaliser of $\pi(g_5^h)$ in $\text{Co}_1$, pull back generators of this normaliser to $\GG$, and adjust the pulled-back generators by elements of $\QQ$ to obtain coset representatives normalising $g_5^h$ in $\GG$. 
Conjugating these elements by $h^{-1}$ yields some subgroup $Y_1$ of $(N^h \cap \GG)^{h^{-1}} < N$, which we seek to extend. 
We find a second $2\mt{B}$-involution $i_2'$ centralising $g_5$ by random search in $N^h \cap \GG$, conjugate $i_2'$ to the central involution in $\GG$ via some $h' \in \MM$, and repeat the process of passing to $\mt{Co}_1$ and pulling back to $\GG$ to produce a subgroup $Y_2$ of $(N^{h'} \cap \GG)^{(h')^{-1}} < N$. 
At this point, we have generated at least a subgroup $Y = \langle Y_1,Y_2 \rangle$ of $N$. 

We then conduct a random search in $Y$ for all involutions $j_2 \in \MM$ that invert $g_5$ and have the property that all involutions in $\langle g_5,j_2 \rangle \cong \mt{D}_{10}$ are of class $2\mt{B}$. 
The $(2\mt{B},2\mt{B},5\mt{B})$ class multiplication coefficient of $\MM$, which can be calculated from the character table of $\MM$ in {\sf GAP}, tells us that there are $3150000$ such involutions. 
We are able to find all of them without having to check whether $Y=N$, and we test each one to determine whether $S = \langle g_2,g_3,j_2 \rangle$ is isomorphic to $P$. 
As in the corresponding calculation in Section~\ref{sec_PSLPGL}, most of the involutions $j_2$ are quickly eliminated; see Remark~\ref{remPSL2(16)}. 
For $A = A_\mt{G}$ and $A = A_\mt{T}$, it is never the case that $S \cong P$. 
For $A = A_\mt{B}$, we find precisely $40$ involutions $j_2$ such that $S \cong P$. 
Moreover, each of these $40$ involutions commutes with the $2\mt{A}$-involution $c_2$ in the ``type $\mt{B}$'' case of Listing~\ref{fig:A5}. 
Given that $c_2$ generates the centraliser of the subgroup $A_\mt{B}$ of $S$ (see Table~\ref{tab_A5}), it follows that $C_\MM(S) = C_\MM(A_\mt{B})$. 
In particular, $S < C_\MM(C_\MM(S)) \cong 2 \udot \BB$, so $S$ is not maximal in $\MM$. 

It remains to show that no almost simple extension of one of the $40$ groups $S = \langle A_\mt{B},j_2 \rangle \cong P$ can be maximal in $\MM$, in the event that such an extension arises in $\MM$. 
Every almost simple extension $E<\MM$ of $S$ normalises $S$, and $N_\MM(S)$ normalises $C_\MM(S)$, so $S \leq E \leq N_\MM(S) \leq N_\MM(C_\MM(S))$. 
As~explained above, $C_\MM(S)$ is generated by a $2\mt{A}$-involution, so $N_\MM(C_\MM(S)) = C_\MM(C_\MM(S)) \cong 2 \udot \BB$. 
In particular, $E$ is certainly not maximal in $\MM$.

\begin{remark} \label{remPSL2(16)}
  \begin{iprf}
    \item Our files containing the involutions $j_2$ are too large to upload to our GitHub repository \cite{ourfile}, but the code given there includes generators for the `large' subgroups $Y$ of $N$, from which the $j_2$ can be recovered by random search; cf. Remark~\ref{remPSL2(13)}. 
To check whether each $S = \langle g_2,g_3,j_2 \rangle$ is isomorphic to $P$, we first checked whether $j_2g_2$ and $j_2g_3$ have orders that arise in $P$. 
This ruled out a vast majority of cases. 
In the remaining cases, we checked whether $|S|=|P|$ by enumerating $S$ as described in Section~\ref{ss:algo}. 
Every group $S$ that passed that test also satisfied $S \cong P$, which we checked by computing the multi-set of element orders in $S$ and applying the Main Theorem of \cite{PSL2qCharacterisation} (which can be verified by constructing all groups of order $4080$ in {\sf GAP} using the {\sf GrpConst} package \cite{grpconst}).
\item 
  It follows from \cite[Proposition~5.1]{W99B} that the projection into $\BB$ of every subgroup $S\cong \PSL_2(16)$ constructed above lies in a maximal subgroup of $\BB$ of shape $2^{9+16}.\mt{PSp}_8(2)$. 
  \end{iprf}
\end{remark}

\section{Proof of Theorem~\ref{thmU34}} \label{sec_U34}

\noindent The group $P = \PSU_3(4)$ has order $2^6{\cdot}3{\cdot}5^2{\cdot}13$, maximal subgroups $(2^2.2^4){:}15$, $5\times \mt{A}_5$, $5^2{:}\mt{S}_3$, and $13{:}3$, all unique up to conjugacy, and a unique class of subgroups isomorphic to $\mt{A}_5$. 
The maximal subgroup $5\times \mt{A}_5$ has $14$ classes of elements $g_5$ of order $5$. 
A direct calculation in {\sf GAP}~\cite{gap} shows that exactly four of these classes are `good' in the sense that there exist involutions $j_2\in P$ such that $\langle g_5,j_2 \rangle \cong \mt{D}_{10}$ and $\langle 5 \times \mt{A}_5, j_2\rangle = P$. 
Every good $g_5 \in 5 \times \mt{A}_5$ projects non-trivially to both direct factors and admits exactly five such $j_2$. 
Moreover, for every element $x_5$ of order $5$ in the direct factor $\mt{A}_5$, there are exactly two elements $y_5$ and $y_5'$ in the cyclic direct factor such that $x_5y_5$ and $x_5y_5'$ are good, and $y_5' = y_5^{-1}$. 
In particular, if we choose any $x_5$ and any $y_5$, then exactly one of $x_5y_5$ and $x_5y_5^2$ is good.

Per Sections~\ref{sec_known_subgroups} and~\ref{secA5}, every as-yet-unclassified subgroup of $\MM$ isomorphic to $P$ has its elements of orders $2$, $3$, and $5$ lying in the $\MM$-classes $2\mt{B}$, $3\mt{C}$, and $5\mt{B}$. 
In particular, every $\mt{A}_5$ in such a subgroup must have type $\mt{G}$, by Table~\ref{tab_A5}. 
Because $A_\mt{G}$ is unique up to conjugacy, it suffices to classify the subgroups of $\MM$ isomorphic to $P$ containing the copy of $A_\mt{G}$ defined in the ``type $\mt{G}$'' case in Listing~\ref{fig:A5}. 
Note that fixing $A_\mt{G}$ also fixes $5 \times A_\mt{G}$, because $C_\MM(A_\mt{G}) \cong \mt{D}_{10}$ has a unique cyclic subgroup of order $5$.

Define $A_\mt{G} = \langle g_2,g_3 \rangle < \GG$ per Listing~\ref{fig:A5}, and consider also the elements $c_2$, $c_5$, $h_u$, and $h_v$ defined there. 
An \textcode{mmgroup} calculation shows that $c_5$ belongs to the $\MM$-class $5\mt{B}$; cf. the proof of Proposition~\ref{prop_A5}. 
Moreover, $c_5$ is inverted by $c_2 \in 2\mt{B}$, so $\langle c_2,c_5 \rangle = C_\MM(A_\mt{G})$.
We therefore have our $5 \times \mt{A}_5$. 
If we now take the element $g_5 = g_2g_3$ of order $5$ in the direct factor $A_\mt{G} \cong \mt{A}_5$, then one of $u_5 = g_5c_5$ and $v_5 = g_5c_5^2$ will be `good' --- assuming that $\langle g_2,g_3,c_5 \rangle \cong 5 \times \mt{A}_5$ extends to $P$ --- but we do not know which one. 
We therefore need to find all involutions that invert $u_5$ and all involutions that invert $v_5$, and test all of them. 
The elements $h_u$ and $h_v$ satisfy $u_5^{h_u} = v_5^{h_v} = g_5$, so the required involutions can be obtained by conjugating the involutions inverting $g_5$ obtained in Section~\ref{sec_L16}. 

We tested each involution $j_2$ that inverts $u_5$ or $v_5$ to see whether it extends $T = \langle g_2,g_3,c_5 \rangle$ to a group isomorphic to $P$. 
This was done by checking the orders of up to $100$ random elements of the group $\langle T,j_2 \rangle$ and discarding $j_2$ if an element of order not arising in $P$ was found. 
This test eliminated all of the involutions inverting $v_5$, and all but five of the involutions inverting $u_5$. 
One of the latter involutions is the element $j_2$ given in Listing~\ref{fig:A5}; the other four lie in $\langle u_5,j_2 \rangle \cong \mt{D}_{10}$. 
An \textcode{mmgroup} calculation confirms that $j_2$ does indeed invert $u_5$. 
To establish that the group $S = \langle T, j_2 \rangle$ is isomorphic to $P$, we check a presentation for $P$. 
The elements $j_2$ and $g_3$ satisfy the following presentation for $P$, the correctness of which can be verified in {\sc GAP} using the functions {\sc IsomorphismFpGroup} and {\sc RelatorsOfFpGroup}:
\[
\langle j_2,g_3 \mid j_2^2=g_3^3=(j_2g_3^{-1}j_2g_3)^5=(j_2g_3)^{15}=((j_2g_3)^3(j_2g_3^{-1})^3)^3=(j_2g_3^{-1}(j_2g_3)^5)^4=1\rangle.
\]
Because $P$ is simple, Von Dyck's Theorem \cite[Theorem~2.53]{handbook} implies that $S \cong P$. 
Moreover, $C_\MM(S)$ is trivial: $C_\MM(T) = \langle c_5 \rangle$, and an \textcode{mmgroup} calculation shows that $c_5$ does not commute with $j_2$. 
Note also that enumerating $S$ as described in Section~\ref{ss:algo} confirms that it has order $62400 = |P|$; see \cite{ourfile}.

We reiterate that $S = \langle g_2,g_3,c_5,j_2 \rangle$ is the unique subgroup of $\MM$ isomorphic to $P \cong \PSU_3(4)$ containing $A_\mt{G} = \langle g_2,g_3 \rangle$, and that, up to conjugacy, it is the unique subgroup of $\MM$ isomorphic to $P$ containing $5\mt{B}$-elements. 
We now show that $S$ extends to a subgroup of $\MM$ isomorphic to the full automorphism group of $P$, i.e. $\PSU_3(4){:}4$. 
By uniqueness of $S$, this also implies that $S$ does {\em not} extend to a {\em maximal} subgroup of $\MM$ isomorphic to $\PSU_3(4){:}2$, as asserted in Theorem~\ref{thmU34}. 

\begin{proposition} \label{prop:U34:4}
The subgroup $U$ of $\MM$ generated by the elements $g_2$, $g_3$, $c_5$, $j_2$, and $a_{12}$ defined in \textcode{mmgroup} format under ``type $\mt{G}$'' in Listing~\ref{fig:A5} is isomorphic to $\PSU_3(4){:}4$ and has trivial centraliser in $\MM$. 
\end{proposition}

\begin{proof}
{Recall that  $S= \langle g_2,g_3,c_5,j_2 \rangle\cong P$. As shown above, $S$} has trivial centraliser in $\MM$, so $N_\MM(S)$ is isomorphic to a subgroup of $\mt{Aut}(P) = \PSU_3(4){:}4$. {An \textcode{mmgroup} calculation shows that $a_{12}$ conjugates each of $g_2$, $g_3$, $c_5$, and $j_2$ into $S$, thus $U\leq N_\MM(S)$. Since $|a_{12}|=12$ and $\PSU_3(4){:}2$ has no elements of order $12$, we  deduce $U \cong \PSU_3(4){:}4$.} 
\end{proof}

\begin{proposition} \label{prop:U34_fusions}
The group $U < \MM$ defined in Proposition~\ref{prop:U34:4} has non-trivial intersection with precisely the conjugacy classes $2\mt{B}$, $3\mt{C}$, $4\mt{C}$, $4\mt{D}$, $5\mt{B}$, $6\mt{F}$, $8\mt{E}$, $10\mt{E}$, $12\mt{J}$, $13\mt{B}$, $15\mt{D}$, $16\mt{B}$, and $16\mt{C}$ of $\MM$. 
The elements of order $4$ in $S<U$ lie in $4\mt{C}$, and the two outer classes of elements of order $4$ lie in $4\mt{D}$. 
Up to conjugacy, $U$ has two cyclic subgroups of order $16$, and each one intersects exactly one of $16\mt{B}$ and $16\mt{C}$.
\end{proposition} 

\begin{proof}
The calculations described here are given in detail in \cite{ourfile}. 
As before, $\chi_\MM$ is the character of the $196883$-dimensional $\mathbb{C}\MM$-module. 
Label $U$-classes by lowercase letters and $\MM$-classes by uppercase letters. 
The group $U$ has $22$ conjugacy classes, labelled $1\mt{a}$, $2\mt{a}$--$\mt{b}$, $3\mt{a}$, $4\mt{a}$--$\mt{c}$, $5\mt{a}$--$\mt{b}$, $6\mt{a}$, $8\mt{a}$--$\mt{b}$, $10\mt{a}$--$\mt{b}$, $12\mt{a}$--$\mt{b}$, $13\mt{a}$, $15\mt{a}$, and $16\mt{a}$--$\mt{d}$ in its character table in {\sf GAP}, constructed as {\sf CharacterTable("U3(4).4")}. 

We already know that $g_3 \in 3\mt{C}$, so $3a$ fuses to $3\mt{C}$. 
The socle $S \cong P$ of $U$ has one class of involutions. 
Comparing centraliser orders in $S$ and $U$ indicates that these involutions comprise the $U$-class $2\mt{a}$. 
We know that $g_2 \in S \cap 2\mt{B}$, so $2\mt{a}$ fuses to $2\mt{B}$. 
All elements of order $12$ in $U$ power to $2\mt{b}$, so $a_{12}^6 \in 2\mt{b}$, and \textcode{chi\_G\_x0()} confirms that $2\mt{b}$ also fuses to $2\mt{B}$. 
The element $a_{12}^2$ commutes with $a_{12}^6$, so we can conjugate it into $\GG$ and apply \textcode{chi\_G\_x0()} to find that its $\chi_\MM$-value is $-1$. 
This shows that $6\mt{a}$ fuses to $6\mt{F}$. 
The socle $S<U$ has a unique class of elements of order $4$; based on centraliser orders, these elements comprise the class $4\mt{a}$. 
The element $j_2g_3^{-1}(j_2g_3)^5 \in S$ has order $4$, and its $\chi_\MM$-value is $19$, so $4\mt{a}$ fuses to $4\mt{C}$. 
The classes $4\mt{b}$--$\mt{c}$ lie outside $S$, and they fuse to a single $\MM$-class because they yield a single class of cyclic subgroups in $U$ and because the $\MM$-classes of order-$4$ elements are rational. 
All elements of order $12$ in $U$ power to $4\mt{b} \cup 4\mt{c}$. 
The $\chi_\MM$-value of $a_{12}^3 \in 4\mt{b} \cup 4\mt{c}$ is $-13$, so $4\mt{b}$--$\mt{c}$ fuse to $4\mt{D}$. 
There is a unique class of cyclic subgroups of order $12$ in $U$, so all order-$12$ elements lie in a single $\MM$-class because all classes of order-$12$ elements in $\MM$ are rational. 
The only order-$12$ elements in $\MM$ that power to both $6\mt{F}$ and $4\mt{C} \cup 4\mt{D}$ are those in $12\mt{J}$, so $12\mt{a}$--$\mt{b}$ fuse to $12\mt{J}$. 
The classes $10\mt{a}$--$\mt{b}$ power to $2\mt{a}$--$\mt{b}$ and $5\mt{a}$--$\mt{b}$, respectively. 
The classes $5\mt{a}$--$\mt{b}$ are distinguished by their centraliser orders; in particular, only a $5\mt{a}$-element is centralised by an element of order $3$. 
Given that $c_5$ commutes with $g_3$, it lies in $5\mt{a}$, and we already know that $c_5 \in 5\mt{B}$. 
The element $g_{10} = c_5^3g_2$ has order $10$ and squares to $c_5$, so $g_{10} \in 10\mt{a}$. 
The $\chi_\MM$-value of $g_{10}$ is $0$, so it lies in $10\mt{E}$. 
The classes $2\mt{a}$--$\mt{b}$ are distinguished similarly: only a $2\mt{b}$-element is centralised by an element of order $3$. 
The element $h_{10} = a_{12}^2c_5$ has order $10$, and the involution $h_{10}^5$ commutes with $g_3^{j_2}[h_{10}^5,g_3^{j_2}]^2$, which has order~$3$. 
Therefore, $h_{10}^5 \in 2\mt{b}$, so $h_{10} \in 10\mt{b}$. 
The $\chi_\MM$-values of $h_{10}$ and $h_{10}^2$ are the same as those of $g_{10}$ and $g_{10}^2$, so $10\mt{b}$ and $5\mt{b}$ also fuse to $10\mt{E}$ and $5\mt{B}$. 
The only elements of order $15$ in $\MM$ that power to both $5\mt{B}$ and $3\mt{C}$ are those in class $15\mt{C}$, so $15\mt{a}$ fuses to $15\mt{C}$. 
The classes $16\mt{a}$--$\mt{b}$ yield a single class of cyclic subgroups of order $16$ in $U$, as do the classes $16\mt{c}$--$\mt{d}$. 
The three $\MM$-classes of order-$16$ elements are rational, so each of the pairs $16\mt{a}$--$\mt{b}$ and $16\mt{c}$--$\mt{d}$ fuses to a single $\MM$-class. 
We are claiming that all order-$16$ elements in $U$ lie in $16\mt{B} \cup 16\mt{C}$, so it suffices to exhibit one element that lies in $16\mt{B}$ and one element that lies in $16\mt{C}$. 
The element $g_{16} = a_{12}j_2$ has order $16$ and $\chi_M$-value $7$, which indicates that it lies in $16\mt{C}$. 
The element $h_{16} = c_5^3a_{12}j_2$ has order $16$ and $\chi_\MM$-value $-1$, so it lies in $16\mt{B}$. 
The elements of order $8$ in $U$ yield a single class of cyclic subgroups of order $8$, and all $\MM$-classes of order-$8$ elements are rational, so $8\mt{a}$--$\mt{b}$ fuse to a single $\MM$-class. 
The square of $g_{16}$ has $\chi_\MM$-value $3$, which indicates that $8\mt{a}$--$\mt{b}$ fuse to $8\mt{E}$. 

It remains to show that $13\mt{a}$ fuses to $13\mt{B}$.  
Note that $g_{13} = j_2g_2g_3^{-1} \in S$ has order $13$. 
Although we have not needed to use them anywhere else in the paper, \textcode{mmgroup} includes functions that allow one to act with elements of $\MM$ on its reducible module $V$ of dimension $196884$ over certain `small' rings, including $\mathbb{F}_3$. 
In particular, one can compute the trace of $g \in \MM$ in this representation by computing the image of a (certain) basis for $V$ under $g$. 
This is sufficient for distinguishing between the classes $13\mt{A}$ and $13\mt{B}$, because their $\chi_\MM$-values $11$ and $-2$ differ modulo $3$. 
We calculated the trace of $g_{13}$ modulo $3$ to be $-1$ by adapting some code provided to us by Gerald H\"ohn \cite{hohn}; see Remark~\ref{rem:13inU}. 
This indicates that $g_{13} \in 13\mt{B}$. 
Our (adapted) code is included in \cite{ourfile}. 
\end{proof}

We now complete the proof of Theorem~\ref{thmU34} by showing that $U \cong \PSU_3(4){:}4$ is not contained in any other maximal subgroup $\MM$; cf. Table~\ref{tab:allmax} and the proof of Proposition~\ref{thm_PGL}. 
Most maximal subgroups $L$ are ruled out because $|U|$ does not divide $|L|$, or because $C_\MM(L) \neq 1 = C_\MM(U)$. 
The remaining candidates are $2^{2} \udot ^{2}\mt{E}_6(2){:}\mt{S}_3$, $3^{1+12}\udot 2\udot \mt{Suz}{:}2$, $\mt{S}_3\times\mt{Th}$, $(3^2{:}2\times\mt{P}\Omega^+_8(3))\udot\mt{S}_4$, and $5^4{:}(3\times 2\udot\mt{PSL}_2(25)){:}2$. 
The last three of these groups contain no elements of order $16$. 
Upon checking the class fusion from the first two groups to $\MM$ in {\sf GAP}, we see that $2^{2} \udot ^{2}\mt{E}_6(2){:}\mt{S}_3$ does not intersect the $\MM$-class $15\mt{D}$, and that $3^{1+12}\udot 2\udot \mt{Suz}{:}2$ does not intersect the $\MM$-classes $16\mt{B}$ or $16\mt{C}$. 
Hence, Proposition~\ref{prop:U34_fusions} implies that $U$ is not contained in $L$ in these cases. 
Therefore, $U$ is a maximal subgroup of $\MM$. 

\begin{remark} \label{rem:13inU}
The code that we used in the final case in the proof of Proposition~\ref{prop:U34_fusions} was adapted (for the sake of brevity) from more general code, kindly provided to us by Gerald H\"ohn \cite{hohn}, that determines the conjugacy class of an element $g$ of $\MM$ in \textcode{mmgroup}, assuming that $|g| \neq 27$. 
We note that H\"ohn's code is based on a paper of Barraclough and Wilson~\cite{BW}. 
Functions for acting on the aforementioned $\MM$-module $V$ are described in the section ``The representation of the Monster group'' in the \textcode{mmgroup} documentation~\cite{sey_python}. 
Note also that, in principle, we could find a $2\mt{B}$-involution commuting with $g_{13}$, conjugate $g_{13}$ into $\GG$, and apply \textcode{chi\_G\_x0()}. 
However, finding such an involution is difficult: $g_{13}$ is self-centralising in $U$, and a random search outside of $U$ has little chance of success. 
(Although we have previously constructed the normaliser a $13\mt{B}$-element as described in Remark~\ref{rem:N13}, this construction was based on {\em starting} with the commuting involution, namely $z \in Z(\GG)$.)
\end{remark}

\section{Proof of Theorem~\ref{thmL28}}\label{sec_L28}

\noindent The group $P = \PSL_2(8)$ has order $2^3{\cdot}3^2{\cdot}7$ and maximal subgroups $2^3{:}7$, $\mt{D}_{14}$, and $\mt{D}_{18}$, all unique up to conjugacy. 
For each element $g_7$ of order $7$ in a fixed maximal $2^3{:}7 < P$, there are exactly $7$ involutions $j_2\in P$ such that $\langle g_7,j_2\rangle\cong \mt{D}_{14}$ and $\langle 2^3{:}7,j_2\rangle=P$. 
As explained in Section~\ref{sec_known_subgroups}, every as-yet-unclassified subgroup of $\MM$ isomorphic to $P$ must have its elements of orders $2$ and $7$ lying in the $\MM$-classes $2\mt{B}$ and $7\mt{B}$. 
(Note that all elements of order $7$ in $P$ fuse to a single $\MM$-class.) 
We therefore aim to construct subgroups $P$ of $\MM$ by first constructing a Borel subgroup $B = 2^3{:}7$ with the correct conjugacy class fusion, and then extending $B$ to $P$ via an involution that extends the cyclic $7$ to a $\mt{D}_{14}$. 

Our first task is to find all $B<\MM$ up to conjugacy. 
Rob Wilson suggested to us that every such subgroup could be found in a conjugate of $\GG \cong \Gx$. 
We begin by proving this.

\begin{lemma} \label{lem_7B}
If $B$ is a subgroup of $\MM$ isomorphic to $2^3{:}7$ that intersects only the $\MM$-classes $2\mt{B}$ and $7\mt{B}$ non-trivially, then, up to conjugacy, $B$ is contained in the maximal subgroup $\GG \cong 2^{1+24}\udot \mt{Co}_1$ of $\MM$. 
\end{lemma}

\begin{proof}
The subgroup $B$ of $\MM$ is contained in the normaliser of its normal subgroup $2^3$. 
The normaliser of the $2^3$ is contained in some $2$-local subgroup of $\MM$ that is maximal with respect to being $2$-local. 
The latter subgroups were classified by Meierfrankenfeld and Shpectorov \cite{MS02,MS03}; they are precisely the maximal subgroups $H_1 = 2\udot\BB$, $H_2 = 2^{2}\udot {}^2\mt{E}_6(2){:}\mt{S}_3$, $H_3 = \GG$, $H_4 = 2^{2+11+22}\udot(\mt{M}_{24}\times\mt{S}_3)$, $H_5 = 2^{3+6+12+18}\udot (\mt{PSL}_3(2)\times 3\mt{S}_6)$, $H_6 = 2^{5+10+20}\udot(\mt{S}_3\times \mt{PSL}_5(2))$, and $H_7 = 2^{10+16}\udot\mt{P}\Omega_{10}^+(2)$.
It suffices to show that, among these subgroups, only $H_3$ contains $7\mt{B}$-elements. 
We first note that $H_i \cap 7\mt{B} = \emptyset$ for $i \in \{1,2\}$ and $H_3 \cap 7\mt{B} \neq \emptyset$, by the conjugacy class fusion data stored in {\sf GAP}~\cite{GAPbc,gap}. 
Next, we recall the following from \cite[Section~4]{MS02}. 
Let $U < \MM$ be an elementary abelian $2$-group that is $2\mt{B}$-pure, i.e. does not intersect $2\mt{A}$. 
The group $U$ is said to be {\em singular} if $U \leq \QQ_u$ for all $u \in U \setminus \{1\}$, where $\QQ_u$ denotes the normal subgroup $2^{1+24}$ of $C_\MM(u) \cong \GG$. 
Up to conjugacy, $\MM$ contains unique $2\mt{B}$-pure singular subgroups $2^2$,  $2^3$, and $2^4$, and two classes of $2\mt{B}$-pure singular subgroups $2^5$. 

The group $H_4$ is the normaliser of a $2\mt{B}$-pure singular subgroup $U \cong 2^2$ of $\MM$. 
By \cite[Lemma~4.5]{MS02}, $H_4/C_\MM(U) \cong \mt{S}_3$, so every $h_7 \in H_4$ of order $7$ lies in $C_\MM(U)$. 
A Sylow $7$-subgroup of $H_4$ has order $7$, so it suffices to exhibit such an element and check that every non-trivial power of it lies in $7\mt{A}$. 
Recall that the central involution $z \in \GG$ is labelled \textcode{M<x\_1000h>} in \textcode{mmgroup}.  
The element $y$ with label \textcode{M<x\_800h>} lies in $\QQ \cap 2\mt{B}$, as does $zy$. 
Hence, $U = \langle z,y \rangle \cong 2^2$ is $2\mt{B}$-pure, and \cite[Lemma~4.1]{MS02} (or a further calculation) implies that $U$ is singular. 
The element \textcode{M<y\_3d6h*x\_125dh*d\_259h*p\_81455268>} has order $7$  and centralises $U$. 
Its non-trivial powers have $\chi_\MM$-value $50$, so they all lie in $7\mt{A}$. 

The group $H_5$ is the normaliser of a $2\mt{B}$-pure singular $2^3$. 
Consider a subgroup of $H_5$ of shape $2^{3+6+12+18}.(7 \times (3^2{:}3))$, where the quotient is the direct product of a Sylow $7$-subgroup of $\PSL_3(2)$ and a Sylow $3$-subgroup of $3\mt{S}_6$. 
The Schur--Zassenhaus Theorem implies that the extension splits, so a Sylow $7$-subgroup of $H_5$ is centralised by a group of order $3^3$. 
The order of the centraliser of a $7\mt{B}$-element in $\MM$ has $3$-part only $3^2$, so it follows that $H_5 \cap 7\mt{B} = \emptyset$. 

The group $H_6$ is the normaliser of a $2\mt{B}$-pure singular $2^5$ of type ``$2$''.  
By \cite[Lemma~4.2]{MS02}, the $2^5$ can be regarded as a faithful irreducible module for the subquotient $\PSL_5(2)$ of $H_6$. 
Every element of order $7$ in $\PSL_5(2)$ normalises a $3$-dimensional subspace of such a module, so every element of order $7$ in $H_6$ normalises a $2\mt{B}$-pure singular $2^3$, hence lies in a conjugate of $H_5$, and therefore in $7\mt{A}$. 

The group $H_7$ is the normaliser of a subgroup $A \cong 2^{10}$ of $\MM$ known as an {\em ark}; see \cite[Section~4.5]{MS02}.
By \cite[Lemma~5.7]{MS02}, the subgroup $A$ admits a non-degenerate quadratic form of $+$ type defined by declaring the singular vectors to be the identity and the $2\mt{B}$-involutions in $A$.  
(Note that the usage of ``singular'' here is not the same as above.)
By \cite[Lemma 5.8]{MS02}, the quotient $G = \mt{P}\Omega_{10}^+(2)$ of $H_7$ acts naturally on $A$. 
Upon constructing the natural module for $G$ in {\sf GAP}, one can check that every element of order $7$ in $G$ centralises a $4$-dimensional subspace containing both singular and non-singular vectors. 
Hence, every element of order $7$ in $H_7$ centralises a $2\mt{A}$-involution, and so lies in a conjugate of $H_1$.
\end{proof}

\begin{proposition} \label{prop_2^3:7}
The elements $e_1$, $e_2$, $e_3$, and $g_7$ of $\MM$ given in \textcode{mmgroup} format in Listing~\ref{fig:2^3:7} satisfy $e_1,e_2,e_3 \in \QQ \cap 2\mt{B}$ and $g_7 \in \GG \cap 7\mt{B}$. 
Every subgroup $B$ of $\MM$ that is isomorphic to $2^3{:}7$ and intersects only the $\MM$-classes $2\mt{B}$ and $7\mt{B}$ non-trivially is conjugate in $\MM$ to one of $B_i = \langle e_i,g_7 \rangle$, $i \in \{1,2,3\}$. 
\end{proposition}

\begin{proof}
An \textcode{mmgroup} calculation confirms the first assertion. 
Hence, up to conjugacy, we can take the cyclic $7$ in $B$ to be generated by $g_7$. 
By Lemma~\ref{lem_7B}, every $B<\MM$ of the desired form can be found, up to conjugacy, in $\GG$. 
We claim that the normal subgroup $E \cong 2^3$ of $B$ is then contained in $\QQ$. 
Consider again the projection $\pi : \GG \rightarrow \mt{Co}_1 < \mt{GL}_{24}(2)$ described in Section~\ref{secQ}. 
Given that elements of order $7$ in $B$ act transitively on the non-trivial elements of $E$, we have either $E < \QQ$ or $E \cap \QQ = \{1\}$.
Assume, for a contradiction, that $E\cap \QQ=\{1\}$. In this case,  $B \cong \pi(B) < \mt{Co}_1$. 
A {\sf GAP} calculation similar to the first calculation in Listing~\ref{figgap1} shows that the $\mt{Co}_1$-classes $7\mt{A}$ and $7\mt{B}$ lift to the $\MM$-classes $7\mt{B}$ and $7\mt{A}$, so $\pi(B)\cong 2^3{:}7$ is a subgroup of $\mt{Co}_1$ containing elements of the $\mt{Co}_1$-class $7\mt{A}$. 
However, by \cite[Lemma~5]{W16}, no such subgroup exists in $\mt{Co}_1$. 
Therefore, $E<\QQ$, as claimed.

The following calculations are documented in \cite{ourfile}.  
Passing to {\sc Magma} \cite{magma} as described in Section~\ref{secQ}, we compute the orbits of $N_{\mt{Co}_1}(\pi(g_7))$ on the set of $1$-dimensional subspaces of $\QQ/Z(\QQ) \cong 2^{24}$. 
(Note that an element of $\mt{Co}_1$-class $7\mt{A}$ preserves a decomposition of $2^{24}$ into eight $3$-dimensional irreducible sub-modules.)
We then lift a representative of each orbit back to an element $x$ of $\QQ$, and check whether either element of the coset $x Z(\QQ)$ extends $\langle g_7 \rangle$ to a group of the form $2^3{:}7$ in which the normal $2^3$ is $2\mt{B}$-pure. 
Exactly three orbits pass this test, and they yield the elements $e_1$, $e_2$, and $e_3$. 
\end{proof}

\begin{remark}
By construction, the groups $B_i$ in Proposition~\ref{prop_2^3:7} are not conjugate in $\GG$. 
We have not attempted to check whether they are conjugate in $\MM$, but this is not needed in order to complete the proof of Theorem~\ref{thmL28}. 
Out of interest, we do note the following. 
The subgroups $2^2$ of each $B_i$ are clearly conjugate in $\MM$, given that they are conjugate under $\langle g_7 \rangle$. 
It turns out that subgroups $2^2$ of {\em distinct} $B_i$ are also conjugate in $\MM$. 
By \cite[Lemmas 7.1 and 7.7]{MS02}, there are exactly three classes of $2\mt{B}$-pure $2^2 < \MM$. 
Up to conjugacy, each such $2^2$ contains $z \in Z(\GG)$, and its $\MM$-class is determined by the $\GG$-class of either of its other two involutions, which can be one of $2\mt{C}$, $2\mt{F}$, or $2\mt{G}$ (in the labelling used in the character table of $\GG$ in {\sf GAP}, which differs from the labelling used in \cite{MS02}). 
The $\GG$-classes $2\mt{C}$, $2\mt{F}$, and $2\mt{G}$ are distinguished by the character $\chi_{299}$ for $\GG$ (see Listing~\ref{fig:mmgroup1}), which evaluates to $299$, $43$, or $11$, respectively. 
Upon conjugating each $B_i$ so that it contains $Z(G) = \langle z \rangle$, and selecting an involution $y \neq z$ in each resulting conjugate, we find that $\chi_{299}(y) = 43$ in all cases.
(Equivalently, $y$ projects to the $\mt{Co}_1$-class $2A$, with centraliser $2^{1+8}.\mt{P}\Omega_8^+(2)$, and $C_\MM(\langle z,y \rangle) \cong (2^9 \times 2^{1+6}).2^{1+8}.2^6\mt{A}_8$.)
\end{remark}

To complete the proof of Theorem~\ref{thmL28}, we attempt to extend each $B_i \cong 2^3{:}7$ to a subgroup of $\MM$ isomorphic to $P = \PSL_2(8)$ by adjoining a $2\mt{B}$-involution $j_2$ that inverts $g_7$. 
All such involutions can be found in $N = N_\MM(\langle g_7 \rangle)$, and the $(2\mt{B},2\mt{B},7\mt{B})$ class multiplication coefficient of $\MM$ indicates that there are $72030$ of them. 
As usual, constructing $N$ in \textcode{mmgroup} is a non-trivial task. 
We constructed generators for a `large' subgroup $Y$ of $N$ by proceeding as described in Remark~\ref{remN7B}, and managed to find all $j_2$ via random search in $Y$ without having to determine whether $Y = N$. 
We then checked each $j_2$ to confirm that none of them extends any $B_i$ to $P$. 
This was done by checking the orders of random elements of $\langle B_i,j_2 \rangle$ and discarding $j_2$ once an element of order not arising in $P$ was found. 

\begin{remark} \label{remN7B}
We had hoped to construct a `large' subgroup of $N_7 = N_\MM(\langle g_7 \rangle) \cong 7^{1+4}{:}(3 \times 2\mt{S}_7)$ via the strategy that we used to construct (part of) the normaliser of a $5\mt{B}$-element; see the fifth paragraph of Section~\ref{sec_L16}. 
The idea there was to find two commuting $2\mt{B}$-involutions centralising the $5\mt{B}$-element, and thereby construct two subgroups of its normaliser by passing to the quotient $\mt{Co}_1$ in two conjugates of $\GG$. 
This strategy failed for $7\mt{B}$-elements because, although all involutions in $N_7$ are of class $2\mt{B}$, there is no pair of commuting involutions that also centralise $g_7$ (and a single such involution produced a subgroup of $N_7$ that was too small for our purposes).  
Instead, we first constructed the normaliser $N_3 \cong 3\udot \mt{Fi}_{24}$ of an element of class $3\mt{A}$, which can be obtained e.g. by powering any element of order $48$. 
The latter is easily found by random selection and also powers to a $2\mt{B}$-involution, so we could construct generators for $N_3$ via the `two commuting involutions' strategy. 
We chose $g_7 \in N_3$ and constructed our subgroup $Y$ of $N_7$ by first generating the normaliser of $g_7$ in $N_3$ using an SLP from the {\sf GAP} package {\sf AtlasRep} associated with the online Atlas~\cite{atlas-web}, 
and then conjugating into $\GG$ using a second commuting $2\mt{B}$-involution and building more of $N_7$ by passing to $\mt{Co}_1$. 
Generators for $Y$ are included in \cite{ourfile}, so the reader wishing to reproduce our proof can recover the $72030$ involutions $j_2$ inverting $g_7$ via random search in $Y$. 
(Our generators for $3\udot\mt{Fi}_{24}$ are also included in \cite{ourfile}.)
\end{remark}

\begin{remark} \label{robL28}
After we finished our work, Rob Wilson made his notes \cite{robL28} dealing with the $\PSL_2(8)$ case publicly available; these notes provide an alternative proof of Theorem~\ref{thmL28}.
\end{remark}


\newpage
\appendix
\section{Code listings} \label{app:listings}

\noindent The various code listings referred to throughout the paper are collected here; cf. Remark~\ref{rem:mmgroup}. 
We briefly comment on the listings containing \textcode{mmgroup} elements. 
Although these objects are cryptic in appearance, each one uniquely defines an element of $\MM$, per the \textcode{mmgroup} documentation~\cite{sey_python}. 
We list several of them explicitly so that the reader wishing to verify our proofs can do so as easily as possible. 
As noted in e.g. Remark~\ref{remPSL2(13)}, some of our more intensive calculations are not easily documented, so additional supplementary code is included in our GitHub repository \cite{ourfile}. 
Given the effort that was required to construct some of our elements (see e.g. Remark~\ref{rem:N13}), we record them here for posterity.


\vspace{0.7cm}

{\footnotesize\begin{lstlisting}[captionpos=b,caption={{\sf GAP} code. 
Output~1 shows that the unique class of elements of order $13$ in $\GG$ lifts to the $\MM$-class $13\mt{B}$. 
Output~2 shows all classes in $\GG$ whose projection to the quotient $\GG/\QQ \cong \text{Co}_1$ lies in the class $3\mt{C}$. 
Output~3 shows that $4\mt{F}$ and $4\mt{G}$ are the only $\GG$-classes of order-$4$ elements whose character values in the $299$-dimensional irreducible representation of $\GG$ over $\mathbb{C}$ are $-13$. Outputs~4 and~5 demonstrate how to calculate class multiplication coefficients from the character table of $\MM$.\\},label=figgap1, frame=lines]   
chM    := CharacterTable("M");;
ch2Co1 := CharacterTable("2^(1+24).Co1");;
chCo1  := CharacterTable("Co1");;

fus   := FusionConjugacyClasses(ch2Co1,chM);;
pos13 := Positions(OrdersClassRepresentatives(ch2Co1),13);;
List(pos13, t-> [ClassNames(ch2Co1)[t],ClassNames(chM)[fus[t]]]);
### Output 1: [["13a","13b"]]

fus := FusionConjugacyClasses(ch2Co1,chCo1);;
nms := List([1..Size(fus)], t-> [ClassNames(ch2Co1)[t],ClassNames(chCo1)[fus[t]]]);;
Filtered(nms, x-> x[2]="3c");
### Output 2: [["6e","3c"],["3c","3c"],["6f","3c"],["12b","3c"]]

pos4 := Positions(OrdersClassRepresentatives(ch2Co1),4);;
c299 := Filtered(Irr(ch2Co1), x-> x[1]=299)[1];;
ClassNames(ch2Co1){Filtered(pos4, x-> c299[x]=-13)};
### Output 3: ["4f","4g"]

ClassNames(chM){[3,14,17,18]}; 
### Output 4: ["2b","6b","6e","6f"]
List([14,17,18], x-> ClassMultiplicationCoefficient(chM,3,3,x));
### Output 5: [14152320,466560,91530]
\end{lstlisting}}

{\footnotesize\begin{lstlisting}[captionpos=b,escapeinside={/*@}{@*/},caption={\textcode{mmgroup} commands, cf.\ \cite{sey_python}. 
The object \textcode{g} represents $g \in \MM$. 
In the final line, $\chi_\MM$ is the complex character of the irreducible $196883$-dimensional representation of $\MM$. 
Per \cite{Conway}, the associated module restricts as $196883 = 299 \oplus 98280 \oplus (24 \otimes 4096)$ to $\GG \cong \Gx$. 
The $299$ and the $98280$ are irreducible modules for $\GG$, the $4096 = 2^{12}$ is an irreducible module for the normal subgroup $\QQ \cong 2^{1+24}$ of $\GG$, and the $24$ is an irreducible module for $2 \udot \mt{Co}_1$. 
The characters $\chi_{299}$, $\chi_{24}$, and $\chi_{4096}$ correspond to the $299$, the $24$, and the $4096$. 
Note that $\chi_{299}$ is also a character for $\mt{Co}_1$, and that the $2$-modular Brauer character for $\mt{Co}_1$ acting on $2^{24}$ can be calculated from $\chi_{24}$.
},label=fig:mmgroup1, frame=lines]
MM("M<1>")               /*@ the identity element of $\MM$ in \textcode{mmgroup} @*/
MM("M<x_1000h>")         /*@ the central $2\mt{B}$-involution $z$ in the maximal subgroup $\GG = C_\MM(z) \cong \Gx$ @*/
MM('r','M')              /*@ produces a random element in $\MM$ @*/ 
MM('r','G_x0')           /*@ produces a random element in $\GG$ @*/
g.in_G_x0()              /*@ decides whether $g\in\GG$@*/
g.in_Q_x0()              /*@ decides whether $g$ lies in the normal subgroup $\QQ\cong 2^{1+24}$ of $\GG$ @*/
g.conjugate_involution() /*@ determines the class of an involution $g$; if $g\in 2\mt{B}$, returns $h\in\MM$ such that $g^h=z$@*/
g.chi_G_x0()             /*@ for $g \in \GG$ (only), returns character values $(\chi_\MM(g),\chi_{299}(g),\chi_{24}(g),\chi_{4096}(g))$@*/
\end{lstlisting}}

{\footnotesize
\begin{lstlisting}[language=python, captionpos=b,frame=lines, caption={Python/\textcode{mmgroup} code that produces a $24\times 24$ matrix describing the action of $g\in \GG \cong \QQ.{\rm Co}_1$ on $\QQ/Z(\QQ) \cong 2^{24}$; see also \cite{ourfile}.\\}, label=fig:mmgroup2]    
def elt_to_24_mat(g):
  mat = []
  for i in range(24):
    x = generators.gen_leech2_op_word_leech2(2**(23-i),g.mmdata,len(g.mmdata),0)
    mat.append([int(d) for d in format((x %2**24), '#026b')[2:]])
  return mat
\end{lstlisting}}


\vspace{0.3cm}

{\footnotesize
\begin{lstlisting}[language=Clean,captionpos=b,texcl=false,frame=lines,caption={Standard generators for $\GG \cong \Gx < \MM$ in  \textcode{mmgroup}; see also \cite{ourfile}.\\}, label=fig:std_ab]    
a = MM("M<y_2feh*x_51h*d_6f2h*p_199553794*l_2*p_1900800*l_2*p_684120>")
b = MM("M<y_32bh*x_0e4h*d_30fh*p_81928987*l_2*p_2880*l_1*p_21312*l_1*p_10455360>")
\end{lstlisting}}


\vspace{0.8cm}

{\footnotesize
\begin{lstlisting}[breaklines=true,language=Clean,captionpos=b,texcl=false,frame=lines,caption={
The elements of $\MM$ described in Propositions~\ref{prop_13a} and \ref{lem_13:6} in \textcode{mmgroup} format; see also \cite{ourfile}. 
The elements $g_{13}$ and $y_6$ lie in the $\MM$-classes $13\mt{A}$ and $6\mt{B}$, and $\langle g_{13},y_6 \rangle \cong 13{:}6$ with the cyclic $6$ acting faithfully on the cyclic $13$. 
The elements $c$ and $d$ generate a group $H \cong \PSL_3(3)$ that commutes with $\langle g_{13},y_6 \rangle$. 
The roles of the elements $x_6,x_3 \in H$ are described in Proposition~\ref{lem_13:6}.\\}, label=fig:std_PGL]  
g13 = MM("M<y_519h*x_0cb8h*d_3abh*p_178084032*l_2*p_2344320*l_2*p_471482*l_1*t_1*l_2*p_2830080*l_2*p_22371347*l_2*t_2*l_1*p_1499520*l_2*p_22779365*l_2*t_1*l_2*p_2597760*l_1*p_11179396*t_1*l_1*p_1499520*l_2*p_85838017*t_2*l_1*p_1499520*l_1*p_64024721*t_2*l_2*p_2386560*l_2*p_21335269>")

y6 = MM("M<y_4fh*x_1331h*d_0d46h*p_79853974*l_2*p_1943040*l_2*p_2398522*t_1*l_2*p_2344320*l_2*p_1858757*l_2*t_1*l_1*p_960*l_2*p_3120*l_2*p_517440*t_2*l_2*p_2597760*l_1*p_12132032*t_2*l_2*p_2880*l_1*p_465840*l_1*p_1565760*t_1*l_2*p_960*l_1*p_63994992*t_1*l_1*p_2027520*l_1*p_50146>")

c = MM("M<y_0fh*x_0bc4h*d_59h*p_207376512*l_2*p_1943040*l_2*p_22272232*l_2*t_1*l_1*p_1499520*l_1*p_22439*l_1*t_1*l_1*p_1394880*l_1*p_21456*l_2*p_4776960*t_2*l_1*p_1499520*l_2*p_53357227*t_1*l_2*p_960*l_2*p_10665792*l_1*p_6086400*t_1*l_2*p_1943040*l_2*p_64043939*t_2*l_2*p_2956800*l_1*p_64017049>")

d = MM("M<y_5a8h*x_0bcdh*d_941h*p_205645390*l_2*p_2830080*l_2*p_8690*l_2*t_1*l_2*p_1900800*l_2*p_10675420*t_1*l_2*p_2597760*l_1*p_42728016*t_2*l_2*p_2597760*l_1*p_10729207*t_1*l_2*p_1985280*l_1*p_21338086*t_2*l_2*p_2597760*l_1*p_21359269*t_1*l_1*p_1499520*l_1*p_42755907>")

x6 = MM("M<y_44eh*x_1906h*d_2d9h*p_173881751*l_1*p_2640000*l_1*p_1925314*l_1*t_1*l_1*p_2999040*l_1*p_2392772*l_1*t_1*l_1*p_1499520*l_1*p_32461673*l_1*t_1*l_2*p_2344320*l_2*p_84794*t_2*l_2*p_2956800*l_1*p_85413707*t_2*l_2*p_1985280*l_1*p_96477721*t_1*l_2*p_1985280*l_1*p_64023741>")

x3 = MM("M<y_492h*x_1fdeh*d_57ch*p_149571126*l_2*p_2787840*l_2*p_12996808*l_2*t_2*l_2*p_1943040*l_2*p_1463363*l_1*t_1*l_2*p_1394880*l_2*p_21312*l_1*p_1524480*t_2*l_2*p_2386560*l_2*p_21468896*t_2*l_1*p_1499520*l_1*p_956371*l_1*t_2*l_1*p_1920*l_2*p_467856*l_2*p_10315200>")
\end{lstlisting}}


\newpage

{\footnotesize
\begin{lstlisting}[breaklines=true,language=Clean,captionpos=b,texcl=false,frame=lines,caption={
Generators $g_{13}$, $g_6$, $i_2$, and $a_{12}$ for a maximal subgroup of $\MM$ isomorphic to $\PSL_2(13){:}2$ in \textcode{mmgroup} format; see also Proposition~\ref{thm_PGL} and \cite{ourfile}. Note that $g_{13}$ is the same element as in Listing~\ref{fig:std_PGL}, and that $g_6=y_6x_6$ with $y_6$ and $x_6$ as in Listing~\ref{fig:std_PGL}.\\}, label=fig:std_PGL213final]  
g13 = MM("M<y_519h*x_0cb8h*d_3abh*p_178084032*l_2*p_2344320*l_2*p_471482*l_1*t_1*l_2*p_2830080*l_2*p_22371347*l_2*t_2*l_1*p_1499520*l_2*p_22779365*l_2*t_1*l_2*p_2597760*l_1*p_11179396*t_1*l_1*p_1499520*l_2*p_85838017*t_2*l_1*p_1499520*l_1*p_64024721*t_2*l_2*p_2386560*l_2*p_21335269>")

g6 = MM("M<y_764h*x_590h*d_0bf6h*p_63465756*l_1*p_24000*l_2*p_528432*t_1*l_2*p_1457280*l_1*p_23214136*l_1*t_2*l_2*p_2344320*l_2*p_13038217*l_2*t_1*l_2*p_2956800*l_1*p_85332887*t_2*l_2*p_2830080*l_2*p_85335745*t_2*l_2*p_1900800*l_2*p_13472*t_2*l_2*p_2386560*l_2*p_85413728*t_1*l_2*p_2386560*l_2*p_53803593>")

i2 = MM("M<y_6ch*x_7ch*d_52ah*p_115885662*l_2*p_2787840*l_2*p_12552610*l_2*t_1*l_2*p_1900800*l_2*p_31998118*l_2*t_2*l_2*p_80762880*l_1*p_243091248*l_2*t_1*l_2*p_2597760*l_1*p_42794439*t_1*l_1*p_1394880*l_2*p_64015152*t_1*l_1*p_2027520*l_1*p_177984*t_1*l_2*p_79432320*l_1*p_161927136>")

a12 = MM("M<y_1afh*x_1661h*d_2ddh*p_208095583*l_2*p_1943040*l_2*p_1974295*l_2*t_2*l_2*p_1900800*l_2*p_10778*l_2*t_2*l_2*p_1900800*l_2*p_1868387*l_1*t_1*l_2*p_2956800*l_1*p_11159238*t_1*l_2*p_1985280*l_1*p_86275805*t_2*l_2*p_2386560*l_2*p_42712609*t_2*l_1*p_1499520*l_1*p_106699812>")
\end{lstlisting}}


\vspace{0.8cm}

{\footnotesize
\begin{lstlisting}[breaklines=true,language=Python,captionpos=b,texcl=false,frame=lines,
caption={Generators for $\mt{A}_5\times\mt{A}_{12} < \MM$ in \textcode{mmgroup} format; see also \cite{ourfile} and the proofs of Propositions~\ref{prop_A12} and~\ref{prop_A5}.}, 
label=fig:A5A12]
# generators for A12
x3 = MM("M<y_31h*x_0d92h*d_85ah*p_240874113*l_1*p_80762880*l_1*p_221802288*t_1*l_2*p_50160000*l_1*p_232003248*l_2*t_2*l_1*p_78988800*l_1*p_182328960*l_1*t_1*l_2*p_118018560*l_1*t_1*l_1*p_183216000*l_1>")

x10 = MM("M<y_491h*x_18h*d_77ah*p_179668320*l_1*p_68344320*l_2*p_159709440*l_2*t_1*l_1*p_70561920*l_2*p_242647728*l_2*t_1*l_1*p_79875840*l_1*p_182772480*l_1*t_1*l_1*p_4012800*l_2*t_1*l_2*p_117575040*l_1>")

# generators for A5 commuting with A12
a2 = MM("M<y_511h*x_19e5h*d_0f88h*p_175676956*l_2*p_127776000*t_2*l_1*p_60360960*l_1*p_193416960*l_2*t_1*l_1*p_69231360*l_2*p_162370608*l_2*t_2*l_1*p_67457280>")

a3 = MM("M<y_411h*x_158eh*d_64fh*p_160702030*l_2*p_1900800*l_2*p_684131*t_1*l_1*p_1499520*l_1*p_32064306*l_2*t_1*l_2*p_1394880*l_1*p_22320*l_2*p_98880*t_2*l_2*p_2830080*l_2*p_21469865*t_2*l_2*p_2830080*l_2*p_106661290*t_1*l_2*p_2597760*l_1*p_43613421*t_2*l_2*p_2830080*l_2*p_96456578>")

# generators for A5 < A12 with orbits of size 6 and 6 on 12 points
b2 = MM("M<y_599h*x_41ah*d_6b7h*p_240430467*l_1*p_70561920*l_1*p_140194560*t_1*l_1*p_81206400*l_2*p_169023408*l_1*t_1*l_2*p_79432320*l_2*p_212044848*l_2*t_1*l_2*p_59917440*l_1*p_157048416>")

b3 = MM("M<y_1eeh*x_15e7h*d_0d65h*p_141989494*l_1*p_59473920*l_2*p_131767728*l_2*t_2*l_2*p_50160000*l_2*p_179224368*l_2*t_2*l_1*p_71005440*l_1*p_243091248*l_1*t_2*l_1*p_58143360*l_2*p_179667936>")
\end{lstlisting}}


\newpage

{\footnotesize
\begin{lstlisting}[breaklines=true,language=Python,captionpos=b,texcl=false,frame=lines,
caption={Generators $g_2$ and $g_3$ for subgroups $A_\mt{G},A_\mt{T},A_\mt{B} \cong \mt{A}_5$ of $\MM$ in \textcode{mmgroup} format; see also \cite{ourfile}. 
In the third case, the $g_i$ are defined in terms of the elements $a_i$ and $b_i$ given in Listing~\ref{fig:A5A12}.
The elements $c_i$ generate the centraliser of $\langle g_2,g_3 \rangle$ in $\MM$. 
In the second and third cases, $i_2$ is a $2\mt{B}$-involution centralising the element $g_2g_3$ of order~$5$, and $h$ is an element conjugating $i_2$ to $z \in Z(\GG)$. 
See also the proof of Proposition~\ref{prop_A5}. 
The roles of the other elements in the ``type $\mt{G}$'' case are explained in Section~\ref{sec_U34}.}, 
label=fig:A5]
# type G
g2 = MM("M<y_4f6h*x_1f98h*d_0b7h*p_67615847*l_1*p_2999040*l_1*p_86264262*l_2*p_11172480>")

g3 = MM("M<y_4e1h*x_19cbh*d_9c8h*p_19643307*l_1*p_2999040*l_1*p_64003504*l_2*p_1478400>")

c2 = MM("M<x_1000h>")

c5 = MM("M<y_548h*x_34ah*d_0a9ch*p_243281095*l_1*p_1457280*l_2*p_43255315*t_2*l_1*p_3840*l_2*p_465936*l_2*p_1101120*t_1*l_2*p_2787840*l_2*p_32009429*l_1*t_2*l_2*p_2956800*l_1*p_64018007*t_1*l_2*p_2880*l_2*p_3120*l_2*p_2579520*t_2*l_2*p_2830080*l_2*p_42706069*t_1*l_2*p_2787840*l_2*p_148289>")

hu = MM("M<y_51h*x_319h*d_0d15h*p_65451314*l_2*p_2344320*l_2*p_23241234*l_1*t_1*l_1*p_2027520*l_1*p_33397753*l_1*t_1*l_2*p_1393920*l_1*p_10666752*l_1*p_3847680*t_2*l_2*p_1943040*l_2*p_21374615*t_1*l_2*p_2880*l_2*p_10666896*l_2*p_2959680*t_1*l_1*p_1499520*l_1*p_85838981*t_1*l_2*p_506880>")

hv = MM("M<y_152h*x_33ah*d_0f3h*p_183194050*l_2*p_71005440*l_2*p_85198272*t_1*l_2*p_2386560*l_2*p_2355444*l_1*t_1*l_2*p_1920*l_2*p_467712*l_2*p_10371840*t_2*l_2*p_2597760*l_1*p_96018816*t_2*l_2*p_1985280*l_1*p_21419925*t_1*l_2*p_1900800*l_2*p_135669*t_2*l_1*p_1499520*l_1*p_64088183>")

j2 = MM("M<y_538h*x_170fh*d_0a21h*p_129570785*l_1*p_1499520*l_2*p_12610123*t_2*l_1*p_2640000*l_1*p_1867433*l_1*t_2*l_2*p_1900800*l_2*p_2798997*l_2*t_2*l_1*p_1499520*l_2*p_127989729*t_2*l_2*p_2830080*l_2*p_21352704*t_2*l_2*p_1985280*l_1*p_127995457*t_1*l_2*p_2597760*l_1*p_64002648*t_1*l_2*p_1943040*l_2*p_96485380>")

a12 = MM("M<y_578h*x_309h*d_1f8h*p_113794596*l_1*p_2027520*l_1*p_10863009*t_1*l_1*p_2999040*l_1*p_1871269*l_2*t_1*l_1*p_1457280*l_2*p_12556433*l_1*t_1*l_1*p_2999040*l_1*p_6743*t_1*l_1*p_1499520*l_1*p_127990661*t_1*l_1*p_1858560*l_1*p_21408*l_2*p_240000*t_2*l_1*p_4797120*l_1*t_2*l_2*p_2830080*l_2*p_106700769>")


# type T
g2 = MM("M<y_82h*x_140eh*d_327h*p_130881367*l_1*p_80319360*l_1*p_131324208*l_1*t_1*l_1*p_69674880*l_2*p_160152960*l_1*t_1*l_1*p_48829440*l_1*p_230229120*l_2*t_1*l_1*p_70561920*l_1*p_87859296>")

g3 = MM("M<y_430h*x_0d4h*d_8a2h*p_242204766*l_2*p_60804480*l_2*p_11552640*l_2*t_1*l_2*p_49272960*l_1*p_172128000*l_2*t_1*l_1*p_59917440*l_1*p_239986560*l_1*t_2*l_2*p_3125760*l_2*t_1*l_2*p_47055360>")

c2 = MM("M<d_200h>")

c3 = MM("M<y_4cdh*x_1274h*d_499h*p_8151915*l_2*p_1900800*l_2*p_43255347*t_2*l_2*p_2597760*l_1*p_479249*l_2*t_2*l_1*p_4654080*t_1*l_2*p_2956800*l_1*p_53436116*t_2*l_2*p_2386560*l_2*p_85412773*t_1*l_1*p_1499520*l_1*p_106661296>")

i2 = MM("M<y_1d9h*x_1d53h*d_170h*p_157936168*l_2*p_68344320*l_2*p_202730880*l_2*t_1*l_1*p_78545280*l_1*p_212044848*l_2*t_2*l_1*p_80762880*l_2*p_149508480*l_2*t_1*l_1*p_81206400*l_1*p_85198176>")

h = MM("M<y_17eh*x_143ah*d_0c93h*p_48068830*l_2*p_2956800*l_1*p_43160055*t_2*l_2*p_1943040*l_2*p_1471043*l_1*t_2*l_1*p_1499520*l_1*p_32513830*l_1*t_1*l_2*p_2830080*l_2*p_85329986*t_2*l_2*p_1985280*l_1*p_96485399*t_1*l_2*p_2386560*l_2*p_85330945>")




# type B
g2, g3 = a2*b2, a3*b3

c2 = MM("M<y_15h*x_1c83h*d_955h*p_191219869*l_2*p_48829440*l_2*p_85198080*t_1*l_2*p_7560960*l_2*p_1795200*t_2*l_1*p_67013760*l_1>")

i2 = MM("M<y_487h*x_1426h*d_602h*p_173036153*l_2*p_47055360*l_1*p_53264640*t_1*l_1*p_60360960*l_1*p_182772480*l_2*t_1*l_1*p_59473920*l_2*p_192086400*l_2*t_1*l_1*p_3569280*l_2*t_2*l_2*p_66570240*l_1>")

h = MM("M<y_4f1h*x_9bch*d_0f77h*p_106507260*l_1*p_80762880*l_2*p_213375504*t_2*l_1*p_1499520*l_2*p_583047*t_2*l_2*p_1900800*l_2*p_1040998*t_2*l_2*p_2386560*l_2*p_21331401*t_1>")
\end{lstlisting}}


\vspace{0.8cm}

{\footnotesize
\begin{lstlisting}[breaklines=true,language=Python,captionpos=b,texcl=false,frame=lines,
caption={Generators for the subgroups $B_i = \langle e_i,g_7 \rangle \cong 2^3{:}7$ of $\GG$ in Proposition~\ref{prop_2^3:7}.}, 
label=fig:2^3:7]
e1 = MM("M<x_1920h*d_4c8h>")
e2 = MM("M<x_1d90h*d_5b0h>")
e3 = MM("M<x_0ec8h*d_3c5h>")
g7 = MM("M<y_5d3h*x_0a6dh*d_8d4h*p_111142481*l_1*p_2999040*l_1*p_43234193>")
\end{lstlisting}}

\end{document}